\documentclass[final]{siamltex}

\usepackage{amssymb}
\usepackage{amsmath}
\usepackage{graphicx}

\usepackage[colorlinks=true,citecolor=black,linkcolor=black]{hyperref}

\newcommand{\bbC}{\mathbb{C}}

\newcommand{\bbR}{\mathbb{R}}

\newcommand{\rmI}{\mathrm{I}}
\newcommand{\rmN}{\mathrm{N}}

\newcommand{\Tr}{\mathrm{Tr}}

\newcommand{\set}[1]{\{{#1}\}}

\newcommand{\norm}[1]{\|{#1}\|}

\newtheorem{problem}[theorem]{Problem}
\newtheorem{example}[theorem]{Example}

\title{Constructing all self-adjoint matrices with prescribed spectrum and diagonal}

\author{Matthew~Fickus\thanks{Department of Mathematics and Statistics, Air Force Institute of Technology, Wright-Patterson Air Force Base, Ohio 45433, Matthew.Fickus@afit.edu.} \and Dustin~G.~Mixon\thanks{Program in Applied and Computational Mathematics, Princeton University, Princeton, New Jersey 08544.} \and Miriam~J.~Poteet\thanks{Department of Mathematics and Statistics, Air Force Institute of Technology, Wright-Patterson Air Force Base, Ohio 45433.} \and Nate Strawn\thanks{Department of Mathematics, Duke University, Durham, North Carolina 27708}}

\begin{document}

\maketitle

\begin{abstract}
The Schur-Horn Theorem states that there exists a self-adjoint matrix with a given spectrum and diagonal if and only if the spectrum majorizes the diagonal.  Though the original proof of this result was nonconstructive, several constructive proofs have subsequently been found.  Most of these constructive proofs rely on Givens rotations, and none have been shown to be able to produce every example of such a matrix.  We introduce a new construction method that is able to do so.  This method is based on recent advances in finite frame theory which show how to construct frames whose frame operator has a given prescribed spectrum and whose vectors have given prescribed lengths.  This frame construction requires one to find a sequence of eigensteps, that is, a sequence of interlacing spectra that satisfy certain trace considerations.  In this paper, we show how to explicitly construct every such sequence of eigensteps.  Here, the key idea is to visualize eigenstep construction as iteratively building a staircase.  This visualization leads to an algorithm, dubbed Top Kill, which produces a valid sequence of eigensteps whenever it is possible to do so.  We then build on Top Kill to explicitly parametrize the set of all valid eigensteps.  This yields an explicit method for constructing all self-adjoint matrices with a given spectrum and diagonal, and moreover all frames whose frame operator has a given spectrum and whose elements have given lengths.
\end{abstract}



%%%%%%%%%%%%%%%%%%%%%%%%%%%%%%%%%%%%%%%%%%%%%%%%%%%%%%%%%%%%%%%%
%%%%%%%%%%%%%%%%%%%%%%%%%%%%%%%%%%%%%%%%%%%%%%%%%%%%%%%%%%%%%%%%
% Section 1: Introduction
%%%%%%%%%%%%%%%%%%%%%%%%%%%%%%%%%%%%%%%%%%%%%%%%%%%%%%%%%%%%%%%%
%%%%%%%%%%%%%%%%%%%%%%%%%%%%%%%%%%%%%%%%%%%%%%%%%%%%%%%%%%%%%%%%

\section{Introduction}

Given nonincreasing sequences $\set{\lambda_n}_{n=1}^{N}$ and $\set{\mu_n}_{n=1}^{N}$, consider the problem of finding an $N\times N$ self-adjoint matrix $G$ which has $\set{\lambda_n}_{n=1}^{N}$ as its spectrum and $\set{\mu_n}_{n=1}^{N}$ as its diagonal entries.  The question of whether or not such a matrix exists is addressed by the classical \textit{Schur-Horn Theorem}.  To be precise, $\set{\lambda_n}_{n=1}^{N}$ is said to \textit{majorize} $\set{\mu_n}_{n=1}^{N}$, denoted $\set{\lambda_n}_{n=1}^N\succeq\set{\mu_n}_{n=1}^{N}$, if
\begin{align}
\label{equation.definition of majorization 1}
\sum_{m=1}^{n}\lambda_{m}
&\geq\sum_{m=1}^{n}\mu_{m}\qquad\forall n=1,\dotsc,N-1,\\
\label{equation.definition of majorization 2}
\sum_{m=1}^{N}\lambda_{m}
&=\sum_{m=1}^{N}\mu_{m}.
\end{align}
Schur~\cite{Schur:23} found that the spectrum of a self-adjoint matrix necessarily majorizes its diagonal entries.  A few decades later, Horn proved the converse~\cite{Horn:54}, yielding:

\textsc{Schur-Horn Theorem}. \textit{There exists a self-adjoint matrix $G$ with spectrum $\set{\lambda_n}_{n=1}^{N}$ and diagonal entries $\set{\mu_n}_{n=1}^{N}$ if and only if $\set{\lambda_n}_{n=1}^N\succeq\set{\mu_n}_{n=1}^{N}$.}

Horn's original proof was nonconstructive.  In subsequent decades, several constructive proofs were found.  In particular, the Chan-Li algorithm~\cite{ChanLi:83} conjugates a given diagonal matrix by a finite number of Givens rotations so as to produce a self-adjoint matrix with a given majorized diagonal.  Related algorithms and their generalizations are considered in~\cite{BendelM:78,DavisH:00,DhillonHST:05}.  Such matrices can also be constructed by an optimization-based limiting process~\cite{Chu:95}.  Alternative algebraic proofs of the Schur-Horn Theorem are given in~\cite{LeiteRT:99}.

In this paper, we provide a new method for constructing Schur-Horn matrices.  In contrast to previous work, this method constructs \textit{all} such matrices.  This method relies on recent developments~\cite{CahillFMPS:11} in the field of \textit{finite frame theory}.  To be precise, the \textit{synthesis operator} of a sequence of vectors $F=\set{f_n}_{n=1}^N$ in $\bbC^M$ is $F:\bbC^N\rightarrow\bbC^M$, $\smash{Fg:=\sum_{n=1}^N g(n)f_n}$.  That is, $F$ is the $M\times N$ matrix whose columns are the $f_n$'s.  Note we make no notational distinction between a sequence of vectors $F=\set{f_n}_{n=1}^{N}$ and the synthesis operator $F$ they induce.  The vectors $F$ are said to be a \textit{frame} for $\bbC^M$ if there exists \textit{frame bounds} $0<A\leq B<\infty$ such that $A\norm{f}^2\leq\norm{F^*f}^2\leq B\norm{f}^2$ for all $f\in\bbC^M$.  The optimal frame bounds $A$ and $B$ of $F$ are the least and greatest eigenvalues of the \textit{frame operator} $FF^*=\sum_{n=1}^{N}f_n^{}f_n^*$, respectively.  As such, $F$ is a frame if and only if the $f_n$'s span $\bbC^M$, which necessitates $M\leq N$.  Broadly speaking, finite frame theory is the study of how to construct $F$ so that $FF^*$ is well-conditioned and so that the $f_n$'s satisfy some additional application-specific, nonlinear constraints.

In particular, over the past decade, much attention was paid to the problem of constructing \textit{unit norm tight frames} (UNTFs), namely frames for which $FF^*=A\rmI$ for some $A>0$ and for which $\norm{f_n}^2=1$ for all $n$.  Such frames yield linear encoders which are optimally robust with respect to additive noise~\cite{GoyalVT:98} and erasures~\cite{CasazzaK:03,HolmesP:04}, and are a generalization of the encoders used in CDMA~\cite{TroppDH:04,ViswanathA:99}.  Unfortunately, such frames are also nontrivial to construct.  Indeed, prior to~\cite{CahillFMPS:11}, only a few explicit examples of such frames where known for any given $M$ and $N$~\cite{CasazzaFMWZ:11,CasazzaL:06,DhillonHST:05,GoyalKK:01}, despite the fact that the set of all such UNTFs contains manifolds of nontrivial dimension when $M>N+1$~\cite{DykemaS:06,Strawn:11}.  Much of the recent work on UNTFs has focused on the \textit{Paulsen problem}~\cite{BodmannC:10,CasazzaFM:11}, a type of Procrustes problem~\cite{Higham:89} concerning how a given frame should be perturbed in order to make it more like a UNTF.  This in turn spurred interest in the following problem:
\begin{problem}
\label{problem.main}
Given any nonnegative nonincreasing sequences $\set{\lambda_m}_{m=1}^{M}$ and $\set{\mu_n}_{n=1}^{N}$, construct all $F=\set{f_n}_{n=1}^{N}$ whose frame operator $FF^*$ has spectrum $\set{\lambda_m}_{m=1}^{M}$ and for which $\norm{f_n}^2=\mu_n$ for all $n$.
\end{problem}

Note that Problem~\ref{problem.main} relates to the Schur-Horn Theorem since the \textit{Gram matrix} $F^*F$ of $F$ has diagonal entries $\set{\mu_n}_{n=1}^{N}$ while the spectrum of $F^*F$ is a zero-padded version of $\set{\lambda_m}_{m=1}^{M}$ when $M\leq N$; this connection is highlighted in~\cite{AntezanaMRS:07,MasseyR:08}.

In this paper, we build on the main results of~\cite{CahillFMPS:11} to provide a complete constructive solution to Problem~\ref{problem.main}, and then use that result to construct all Schur-Horn matrices.  In particular, \cite{CahillFMPS:11} makes use of an observation, nicely explained in~\cite{HornJ:85}, that majorization is simply the end result of the repeated application of a more basic idea: eigenvalue interlacing.  Specifically, a sequence $\set{\beta_m}_{m=1}^{M}$ \textit{interlaces} on another sequence $\set{\alpha_m}_{m=1}^{M}$, denoted $\set{\alpha_m}_{m=1}^{M}\sqsubseteq\set{\beta_m}_{m=1}^{M}$, provided $\alpha_M\leq\beta_M$ and $\beta_{m+1}\leq\alpha_m\leq\beta_m$ for all $m=1,\dotsc,M-1$.  Interlacing naturally arises in the context of frame theory by considering partial sums of the frame operator $FF^*$.  In particular, for any $n=1,\dotsc,N$, the frame operator of the $M\times n$ synthesis operator $F_n$ of the partial sequence of vectors $F_n=\set{f_m}_{m=1}^n$ is
\begin{equation}
\label{equation.partial frame operator}
F_n^{}F_n^*=\sum_{m=1}^{n}f_m^{}f_m^*.
\end{equation}
Letting $\set{\lambda_{n;m}}_{m=1}^{M}$ denote the spectrum of $F_n^{}F_n^*$, a classical result~\cite{HornJ:85} implies that the spectrum $\set{\lambda_{n+1;m}}_{m=1}^{M}$ of $F_{n+1}^{}F_{n+1}^*=F_n^{}F_n^*+f_{n+1}^{}f_{n+1}^*$ interlaces on $\set{\lambda_{n;m}}_{m=1}^{M}$.  Moreover, if $\norm{f_n}^2=\mu_n$ for all $n$, then $\set{\lambda_{n;m}}_{m=1}^{M}$ must also satisfy
\begin{equation}
\label{equation.trace condition}
\sum_{m=1}^M\lambda_{n;m}
=\Tr(F_n^{}F_n^*)
=\Tr(F_n^*F_n^{})
=\sum_{m=1}^{n}\norm{f_m}^2
=\sum_{m=1}^n\mu_m.
\end{equation}
In~\cite{CahillFMPS:11}, a sequence of interlacing spectra which satisfy the trace conditions \eqref{equation.trace condition} is called a sequence of \textit{eigensteps}; in this paper, we call them \textit{outer eigensteps}:
\begin{definition}
\label{definition.outer eigensteps}
Let $\set{\lambda_m}_{m=1}^{M}$ and $\set{\mu_n}_{n=1}^{N}$ be nonnegative and nonincreasing.  A corresponding sequence of \textbf{outer eigensteps} is a sequence of sequences $\set{\set{\lambda_{n;m}}_{m=1}^{M}}_{n=0}^{N}$ which satisfies the following four properties:
\begin{enumerate}
\renewcommand{\labelenumi}{(\roman{enumi})}
\item $\lambda_{0;m}=0$ for every $m=1,\ldots,M$,
\item $\lambda_{N;m}=\lambda_m$ for every $m=1,\ldots,M$,
\item $\set{\lambda_{n-1;m}}_{m=1}^{M}\sqsubseteq\set{\lambda_{n;m}}_{m=1}^{M}$ for every $n=1,\ldots,N$,
\item $\sum_{m=1}^{M}\lambda_{n;m}=\sum_{m=1}^{n}\mu_{m}$ for every $n=1,\ldots,N$.
\end{enumerate}
\end{definition}
In light of the above discussion, any $F=\set{f_n}_{n=1}^{N}$ for which $FF^*$ has $\set{\lambda_m}_{m=1}^{M}$ as its spectrum and for which $\norm{f_n}^2=\mu_n$ for all $n$ generates a sequence of outer eigensteps.  The main result of~\cite{CahillFMPS:11} proves that the converse is also true:
\begin{theorem}[Theorem~$2$ of \cite{CahillFMPS:11}]
\label{theorem.necessity and sufficiency of eigensteps}
For any nonnegative nonincreasing sequences $\set{\lambda_m}_{m=1}^{M}$ and $\set{\mu_n}_{n=1}^{N}$, every sequence of vectors $F=\set{f_n}_{n=1}^{N}$ in $\bbC^M$ whose frame operator $FF^*$ has spectrum $\set{\lambda_m}_{m=1}^{M}$ and which satisfies $\norm{f_n}^2=\mu_n$ for all $n$ can be constructed by the following process:
\begin{enumerate}
\renewcommand{\labelenumi}{\Alph{enumi}.}
\item
Pick outer eigensteps $\set{\set{\lambda_{n;m}}_{m=1}^{M}}_{n=0}^{N}$ as in Definition~\ref{definition.outer eigensteps}. 
\item
For each $n=1,\dotsc,N$, consider the polynomial:
\begin{equation*}
p_n(x):=\prod_{m=1}^{M}(x-\lambda_{n;m}).
\end{equation*}
Take any $f_1\in\bbC^M$ such that $\norm{f_1}^2=\mu_1$.

For each $n=1,\dotsc,N-1$, choose any $f_{n+1}$ such that:
\begin{equation}
\label{equation.necessity and sufficiency of eigensteps 2}
\norm{P_{n;\lambda}f_{n+1}}^2=-\lim_{x\rightarrow\lambda}(x-\lambda)\frac{p_{n+1}(x)}{p_n(x)}\qquad \forall \lambda\in\set{\lambda_{n;m}}_{m=1}^M.
\end{equation}
Here, $P_{n;\lambda}$ denotes the orthogonal projection operator onto the eigenspace $\rmN(\lambda\rmI-F_n^{}F_n^*)$ of the frame operator~\eqref{equation.partial frame operator} of $F_n:=\set{f_m}_{m=1}^{n}$.\\
The limit in~\eqref{equation.necessity and sufficiency of eigensteps 2} necessarily exists and is nonpositive.
\end{enumerate}
Conversely, any $F$ constructed by this process has $\set{\lambda_m}_{m=1}^{M}$ as the spectrum of $FF^*$ and $\norm{f_n}^2=\mu_n$ for all $n$, and moreover, $F_n^{}F_n^*$ has spectrum $\set{\lambda_{n;m}}_{m=1}^{M}$.
\end{theorem}

We emphasize that Theorem~\ref{theorem.necessity and sufficiency of eigensteps} is proven from basic principles in~\cite{CahillFMPS:11}, the key idea being to write $p_{n+1}(x)$ in terms of $p_{n}(x)$, a fact also recently exploited in~\cite{BatsonSS:11}.  In particular, the proof of Theorem~\ref{theorem.necessity and sufficiency of eigensteps} does not rely on the Schur-Horn Theorem.  We further note that, although Theorem~\ref{theorem.necessity and sufficiency of eigensteps} provides an answer to Problem~\ref{problem.main}, this answer is incomplete.  To be clear, Step~B involves only standard algebraic techniques, and it can be made surprisingly explicit; see Theorem~7 of~\cite{CahillFMPS:11}.  In fact, \cite{FickusMP:11} provides MATLAB code to implement Step~B.  Step~A, on the other hand, is vague: for a given $\set{\lambda_m}_{m=1}^{M}$ and $\set{\mu_n}_{n=1}^{N}$ it is unclear how to construct a single valid sequence of outer eigensteps, much less find them all.  The techniques of this paper will make Step~A explicit, with our main result being:
\begin{theorem}
\label{theorem.Step A redone}
Let $\set{\lambda_m}_{m=1}^{M}$ and $\set{\mu_n}_{n=1}^{N}$ be nonnegative and nonincreasing where $M\leq N$.  There exists a sequence of vectors $F=\set{f_n}_{n=1}^{N}$ in $\bbC^M$ whose frame operator $FF^*$ has spectrum $\set{\lambda_m}_{m=1}^{M}$ and for which $\norm{f_n}^2=\mu_n$ for all $n$ if and only if $\set{\lambda_m}_{m=1}^{M}\cup\{0\}_{m=M+1}^N\succeq\set{\mu_n}_{n=1}^{N}$.  Moreover, if $\set{\lambda_m}_{m=1}^{M}\cup\{0\}_{m=M+1}^N\succeq\set{\mu_n}_{n=1}^{N}$, then every such $F$ can be constructed by the following process:
\begin{enumerate}
\renewcommand{\labelenumi}{\Alph{enumi}.}
\item
Let $\{\lambda_{N;m}\}_{m=1}^M:=\{\lambda_{m}\}_{m=1}^M$.\\
For $n=N,\ldots,2$, construct $\{\lambda_{n-1;m}\}_{m=1}^M$ in terms of $\set{\lambda_{n;m}}_{m=1}^{M}$ as follows:\\
\mbox{\quad}For each $k=M,\ldots,1$, if $k>n-1$, take $\lambda_{n-1;k}:=0$.\\
\mbox{\quad}Otherwise, pick any $\lambda_{n-1;k}\in[A_{n-1;k},B_{n-1;k}]$, where
\begin{align*}
A_{n-1;k}&:=\max \bigg\{ \lambda_{n;k+1}, \sum_{m=k}^{M} \lambda_{n;m} - \sum_{m=k+1}^{M} \lambda_{n-1;m} - \mu_n \bigg\},\\
B_{n-1;k}&:=\min \bigg\{ \lambda_{n;k}, \min_{{l}=1,\dots,k} \bigg\{\sum_{m={l}}^{n-1} \mu_m - \sum_{m={l}+1}^{k} \lambda_{n;m} - \sum_{m=k+1}^{M} \lambda_{n-1;m} \bigg\}\bigg\}.
\end{align*}
Here, we use the convention that $\lambda_{n;M+1}=0$, and that sums over empty sets of indices are zero.
\item
Follow Step~B of Theorem~\ref{theorem.necessity and sufficiency of eigensteps}.
\end{enumerate}
Conversely, any $F$ constructed by this process has $\set{\lambda_m}_{m=1}^{M}$ as the spectrum of $FF^*$ and $\norm{f_n}^2=\mu_n$ for all $n$, and moreover, $F_n^{}F_n^*$ has spectrum $\set{\lambda_{n;m}}_{m=1}^{M}$.
\end{theorem}

In the next section, we discuss how solving Problem~\ref{problem.main} via Theorem~\ref{theorem.Step A redone} suffices to construct all Schur-Horn matrices.  We also introduce an alternative notion of eigensteps: whereas outer eigensteps give the spectra of the partial frame operators $F_n^{}F_n^{*}$, \textit{inner eigensteps} will give the spectra of the partial Gram matrices $F_n^*F_n^{}$.  It turns out that this second notion of eigensteps simplifies the needed analysis.  In Section $3$, we then visualize the inner eigenstep construction problem in terms of iteratively building a staircase.  This visualization suggests a new algorithm, dubbed \textit{Top Kill}, which produces a valid sequence of eigensteps whenever it is possible to do so.  In the fourth section, we further exploit the intuition behind Top Kill to find an explicit parametrization of the set of all valid inner eigensteps, leading to a proof of Theorem~\ref{theorem.Step A redone} and thus an explicit construction of all Schur-Horn matrices.

%%%%%%%%%%%%%%%%%%%%%%%%%%%%%%%%%%%%%%%%%%%%%%%%%%%%%%%%%%%%%%%%
%%%%%%%%%%%%%%%%%%%%%%%%%%%%%%%%%%%%%%%%%%%%%%%%%%%%%%%%%%%%%%%%
% Section 2: Preliminaries
%%%%%%%%%%%%%%%%%%%%%%%%%%%%%%%%%%%%%%%%%%%%%%%%%%%%%%%%%%%%%%%%
%%%%%%%%%%%%%%%%%%%%%%%%%%%%%%%%%%%%%%%%%%%%%%%%%%%%%%%%%%%%%%%%

\section{Preliminaries}

In this section, we further detail the connection between the Schur-Horn Theorem and Problem~\ref{problem.main}, and then we reformulate Step~A of Theorem~\ref{theorem.necessity and sufficiency of eigensteps} in terms of an alternative but equivalent notion of eigensteps, dubbed \textit{inner eigensteps}.  With regards to the first point, this connection stems from letting the Schur-Horn matrix $G$ be the Gram matrix $F^*F$ of the sequence of vectors $F=\set{f_n}_{n=1}^{N}$.  

To be precise, given nonnegative nonincreasing sequences $\set{\lambda_m}_{m=1}^{M}$ and $\set{\mu_n}_{n=1}^{N}$ where $M\leq N$, the Schur-Horn Theorem implies that Problem~\ref{problem.main} is feasible if and only if $\set{\mu_n}_{n=1}^{N}$ is majorized by $\set{\lambda_m}_{m=1}^{M}$ padded with $N-M$ zeros.  Indeed, if Problem~\ref{problem.main} has a solution $F$, then $G=F^*F$ has spectrum $\set{\lambda_m}_{m=1}^{M}\cup\set{0}_{m=M+1}^{N}$ and diagonal $\set{\mu_n}_{n=1}^{N}$, and so $\set{\mu_n}_{n=1}^{N}\preceq\set{\lambda_m}_{m=1}^{M}\cup\set{0}_{m=M+1}^{N}$ by the Schur-Horn Theorem.   Conversely, if $\set{\mu_n}_{n=1}^{N}\preceq\set{\lambda_m}_{m=1}^{M}\cup\set{0}_{m=M+1}^{N}$, then the corresponding Schur-Horn matrix $G$ can be unitarily diagonalized:
\begin{equation*}
G=VDV^*=\begin{bmatrix}V_1&V_2\end{bmatrix}\begin{bmatrix}D_1&0\\0&0\end{bmatrix}\begin{bmatrix}V_1^*\\V_2^*\end{bmatrix}=V_1^{}D_1^{}V_1^*,
\end{equation*}
where $D_1$ is an $M\times M$ diagonal matrix with diagonal $\set{\lambda_m}_{m=1}^{M}$; the matrix \smash{$F=D_1^{\frac12}V_1^*$} is then one solution to Problem~\ref{problem.main}.  This line of reasoning is well-known~\cite{AntezanaMRS:07,DhillonHST:05}.

In this paper, we follow an alternative approach that is modeled on that of~\cite{HornJ:85}: rather than use the Schur-Horn Theorem to determine the feasibility of Problem~\ref{problem.main}, we instead independently find all solutions to Problem~\ref{problem.main}, see Theorem~\ref{theorem.Step A redone}, and then use these matrices to construct all Schur-Horn matrices.  To be precise, note that though the Schur-Horn Theorem applies to all self-adjoint matrices $G$, it suffices to consider the case where $G$ is positive semidefinite.  Indeed, any self-adjoint matrix $\hat{G}$ can be written as $\hat{G}=G+\alpha\rmI$ where $G$ is positive semidefinite and $\alpha\leq\lambda_{\min}(\hat{G})$; it is straightforward to show that the spectrum $\set{\hat{\lambda}_n}_{n=1}^{N}$ of $\hat{G}$ majorizes its diagonal $\set{\hat{\mu}_n}_{n=1}^{N}$ if and only if the spectrum $\set{\lambda_n}_{n=1}^{N}=\set{\hat{\lambda}_n-\alpha}_{n=1}^{N}$ of $G$ majorizes its diagonal $\set{\mu_n}_{n=1}^{N}=\set{\hat{\mu}_n-\alpha}_{n=1}^{N}$.  Moreover, since $G$ is positive semidefinite, it has a Cholesky factorization $G=F^*F$ where $F\in\bbC^{N\times N}$.  Regarding $F$ as the synthesis operator of some sequence of vectors $\set{f_n}_{n=1}^{N}$ in $\bbC^N$, we are thus reduced to Problem~\ref{problem.main} in the special case where $M=N$.  Presuming for the moment that Theorem~\ref{theorem.Step A redone} is true, we summarize the above discussion as follows:
\begin{theorem}
\label{theorem.self-adjoint to positive semidefinite}
Given nonincreasing sequences $\set{\hat{\lambda}_n}_{n=1}^{N}$ and $\set{\hat{\mu}_n}_{n=1}^{N}$ such that $\set{\hat{\lambda}_n}_{n=1}^{N}\succeq\set{\hat{\mu}_n}_{n=1}^{N}$, every matrix $\hat{G}$ with spectrum $\set{\hat{\lambda}_n}_{n=1}^{N}$ and diagonal $\set{\hat{\mu}_n}_{n=1}^{N}$ can be constructed as $\hat{G}=F^*F+\alpha\rmI$ where $F$ is any matrix constructed by taking any $\alpha\leq\lambda_{\min}(\hat{G})$ and applying Theorem~\ref{theorem.Step A redone} where $\lambda_n:=\hat{\lambda}_n-\alpha$ and $\mu_n:=\hat{\mu}_n-\alpha$.  Moreover, any $\hat{G}$ constructed in this fashion has the desired spectrum and diagonal.
\end{theorem}

We are thus reduced to solving Problem~\ref{problem.main}, that is, proving Theorem~\ref{theorem.Step A redone}; this problem is the focus of the remainder of this paper.  In light of Theorem~\ref{theorem.necessity and sufficiency of eigensteps}, solving Problem~\ref{problem.main} boils down to finding every valid sequence of outer eigensteps $\set{\set{\lambda_{n;m}}_{m=1}^{M}}_{n=0}^{N}$, see Definition~\ref{definition.outer eigensteps}, for any given nonnegative nonincreasing sequences $\set{\lambda_m}_{m=1}^{M}$ and $\set{\mu_n}_{n=1}^{N}$.  We now briefly summarize an example of such an eigenstep characterization problem given in~\cite{CahillFMPS:11}:
\begin{example}
\label{example.5 in 3 outer}
Consider the problem of constructing all $3\times 5$ matrices $F$ with unit norm columns such that $FF^*=\frac53\rmI$, namely Problem~\ref{problem.main} in the special case where $M=3$, $N=5$, $\lambda_1=\lambda_2=\lambda_3=\frac53$ and $\mu_1=\mu_2=\mu_3=\mu_4=\mu_5=1$.  In~\cite{CahillFMPS:11}, it is shown that every valid sequence of corresponding outer eigensteps $\set{\set{\lambda_{n;m}}_{m=1}^{3}}_{n=0}^{5}$ is of the form:
\begin{equation}
\label{equation.5 in 3 outer 1}
\begin{tabular}{p{1cm} p{1cm} p{1cm} p{1cm} p{1cm} p{1cm} l}
\ \,$n$&$0$&$1$&$2$&$3$&$4$&$5${\smallskip}\\ 
\hline\noalign{\smallskip}
$\lambda_{n;3}$&$0$ &$0$ &$0$ 	&$x$ &$\frac23$ &$\frac{5}3${\smallskip}\\
$\lambda_{n;2}$&$0$ &$0$ &$y$ 	&$\frac{4}3-x$ &$\frac{5}3$ &$\frac{5}3${\smallskip}\\
$\lambda_{n;1}$&$0$ &$1$ &$2-y$ &$\frac{5}3$ &$\frac{5}3$ &$\frac{5}3$\\
\end{tabular}
\end{equation}
where $x$ and $y$ are restricted so as to satisfy the interlacing requirements (iii) of Definition~\ref{definition.outer eigensteps}.  Specifically, $x$ and $y$ must satisfy the eleven inequalities:
\begin{align}
\nonumber
\{\lambda_{3,m}\}_{m=1}^3\sqsubseteq\{\lambda_{4,m}\}_{m=1}^3&\quad\Longleftrightarrow\quad x\leq\tfrac23\leq\tfrac43-x\leq\tfrac53,\\
\label{equation.5 in 3 outer 2}
\{\lambda_{2,m}\}_{m=1}^3\sqsubseteq\{\lambda_{3,m}\}_{m=1}^3&\quad\Longleftrightarrow\quad 0\leq x\leq y\leq\tfrac43-x\leq2-y\leq\tfrac53,\\
\nonumber
\{\lambda_{1,m}\}_{m=1}^3\sqsubseteq\{\lambda_{2,m}\}_{m=1}^3&\quad\Longleftrightarrow\quad 0\leq y\leq 1\leq 2-y,
\end{align}
which can be simplified to $0\leq x\leq\frac23$, $\max\set{\frac13,x}\leq y\leq\min\set{\frac23+x,\frac43-x}$.  Choosing any such $(x,y)$ completes Step~A of the algorithm of Theorem~\ref{theorem.necessity and sufficiency of eigensteps}; the corresponding eigensteps~\eqref{equation.5 in 3 outer 1} are then used in Step~B to produce a $3\times 5$ matrix $F$.  For example, if $(x,y)=(0,\frac13)$, one particular implementation of Step~B~\cite{CahillFMPS:11} yields the matrix:
\begin{equation*}
F=\begin{bmatrix}1&\frac23&-\frac1{\sqrt6} & -\frac16 &\frac16\smallskip\\0&\frac{\sqrt5}3&\frac{\sqrt5}{\sqrt6}&\frac{\sqrt5}6&-\frac{\sqrt5}6\smallskip\\0&0&0&\frac{\sqrt5}{\sqrt6}&\frac{\sqrt5}{\sqrt6}\end{bmatrix}.
\end{equation*}
\end{example}

The previous example highlights the key obstacle in using Theorem~\ref{theorem.necessity and sufficiency of eigensteps} to solve Problem~\ref{problem.main}: finding all valid sequences of eigensteps~\eqref{equation.5 in 3 outer 1} often requires reducing a large system of linear inequalities~\eqref{equation.5 in 3 outer 2}.  In the following sections, we provide an efficient method for reducing such systems.  It turns out that this method is more easily understood in terms of an alternative but equivalent notion of eigensteps.  To be clear, for any given sequence of outer eigensteps $\set{\set{\lambda_{n;m}}_{m=1}^{M}}_{n=0}^{N}$, recall from Theorem~\ref{theorem.necessity and sufficiency of eigensteps} that for any $n=1,\dotsc,N$, the sequence $\set{\lambda_{n;m}}_{m=1}^{M}$ is the spectrum of the $M\times M$ frame operator~\eqref{equation.partial frame operator} of the $n$th partial sequence $F_n=\set{f_{m}}_{n=1}^{n}$.  In the theory that follows, it is more convenient to instead work with the spectrum $\set{\lambda_{n;m}}_{m=1}^{n}$ of the corresponding $n\times n$ Gram matrix $F_n^*F_n^{}$; we use the same notation for both spectra since $\set{\lambda_{n;m}}_{m=1}^{n}$ is a zero-padded version of $\set{\lambda_{n;m}}_{m=1}^{M}$ or vice versa, depending on whether $n>M$ or $n\leq M$.   We refer to the values $\set{\set{\lambda_{n;m}}_{m=1}^{n}}_{n=1}^{N}$ as a sequence of \textit{inner eigensteps} since they arise from matrices of inner products of the $f_n$'s (Gram matrices), whereas outer eigensteps $\set{\set{\lambda_{n;m}}_{m=1}^{M}}_{n=0}^{N}$ arise from sums of outer products of the $f_n$'s (frame operators).  To be precise:
\begin{definition}
\label{definition.inner eigensteps}
Let $\set{\lambda_n}_{n=1}^{N}$ and $\set{\mu_n}_{n=1}^{N}$ be nonnegative nonincreasing sequences.  A corresponding sequence of \textbf{inner eigensteps} is a sequence of sequences $\set{\set{\lambda_{n;m}}_{m=1}^{n}}_{n=1}^{N}$ which satisfies the following three properties:
\begin{enumerate}
\renewcommand{\labelenumi}{(\roman{enumi})}
\item $\lambda_{N;m}=\lambda_m$ for every $m=1,\ldots,N$,
\item $\set{\lambda_{n-1;m}}_{m=1}^{n-1}\sqsubseteq\set{\lambda_{n;m}}_{m=1}^{n}$ for every $n=2,\dotsc,N$,
\item $\sum_{m=1}^{n}\lambda_{n;m}=\sum_{m=1}^{n}\mu_{m}$ for every $n=1,\dotsc,N$.
\end{enumerate}
\end{definition}
To clarify, unlike the outer eigensteps of Definition~\ref{definition.outer eigensteps}, the interlacing relation (ii) here involves two sequences of different length; we write $\set{\alpha_m}_{m=1}^{n-1}\sqsubseteq\set{\beta_m}_{m=1}^{n}$ if $\beta_{m+1}\leq\alpha_m\leq\beta_m$ for all $m=1,\ldots,n-1$.  As the next example illustrates, inner and outer eigensteps can be put into correspondence with each other:
\begin{example}
\label{example.5 in 3}
We revisit Example~\ref{example.5 in 3 outer}.  Here, we pad $\set{\lambda_m}_{m=1}^{3}$ with two zeros so as to match the length of $\set{\mu_n}_{n=1}^{5}$.  That is, $\lambda_1=\lambda_2=\lambda_3=\frac53$, $\lambda_4=\lambda_5=0$, and $\mu_1=\mu_2=\mu_3=\mu_4=\mu_5=1$.  We find every sequence of inner eigensteps $\set{\set{\lambda_{n;m}}_{m=1}^{n}}_{n=1}^{5}$, namely every table of the form:
\begin{equation}
\label{equation.5 in 3 inner 0}
\begin{tabular}{p{1cm} p{1cm} p{1cm} p{1cm} p{1cm} l}
\ \,$n$&$1$&$2$&$3$&$4$&$5${\smallskip}\\ 
\hline\noalign{\smallskip}
$\lambda_{n;5}$	&	&	&	&	&$0$		{\smallskip}\\
$\lambda_{n;4}$	&	&	&	&?	&$0$		{\smallskip}\\
$\lambda_{n;3}$	&	&	&?	&?	&$\frac53$	{\smallskip}\\
$\lambda_{n;2}$	&	&?	&?	&?	&$\frac53$	{\smallskip}\\
$\lambda_{n;1}$	&?	&?	&?	&?	&$\frac53$
\end{tabular}
\end{equation}
that satisfies the interlacing properties (ii) and trace conditions (iii) of Definition~\ref{definition.inner eigensteps}.  To be precise, (ii) gives us $0=\lambda_{5;5}\leq\lambda_{4;4}\leq\lambda_{5;4}=0$ and so $\lambda_{4;4}=0$.  Similarly, $\frac53\leq\lambda_{5;3}\leq\lambda_{4;2}\leq\lambda_{3;1}\leq\lambda_{4;1}\leq\lambda_{5;1}=\frac53$ and so $\lambda_{4;2}=\lambda_{3;1}=\lambda_{4;1}=\frac53$, yielding:
\begin{equation}
\label{equation.5 in 3 inner 1}
\begin{tabular}{p{1cm} p{1cm} p{1cm} p{1cm} p{1cm} l}
\ \,$n$&$1$&$2$&$3$&$4$&$5${\smallskip}\\ 
\hline\noalign{\smallskip}
$\lambda_{n;5}$	&	&	&			&			&$0$		{\smallskip}\\
$\lambda_{n;4}$	&	&	&			&$0$		&$0$		{\smallskip}\\
$\lambda_{n;3}$	&	&	&?			&?			&$\frac53$	{\smallskip}\\
$\lambda_{n;2}$	&	&?	&?			&$\frac53$	&$\frac53$	{\smallskip}\\
$\lambda_{n;1}$	&?	&?	&$\frac53$	&$\frac53$	&$\frac53$
\end{tabular}
\end{equation}
Meanwhile, since $\mu_m=1$ for all $m$, the trace conditions (iii) give that the values in the $n$th column of~\eqref{equation.5 in 3 inner 1} sum to $n$.  Thus, $\lambda_{1;1}=1$ and $\lambda_{4;3}=\frac23$:
\begin{equation*}
\begin{tabular}{p{1cm} p{1cm} p{1cm} p{1cm} p{1cm} l}
\ \,$n$&$1$&$2$&$3$&$4$&$5${\smallskip}\\ 
\hline\noalign{\smallskip}
$\lambda_{n;5}$	&	&	&			&			&$0$		{\smallskip}\\
$\lambda_{n;4}$	&	&	&			&$0$		&$0$		{\smallskip}\\
$\lambda_{n;3}$	&	&	&?			&$\frac23$	&$\frac53$	{\smallskip}\\
$\lambda_{n;2}$	&	&?	&?			&$\frac53$	&$\frac53$	{\smallskip}\\
$\lambda_{n;1}$	&$1$&?	&$\frac53$	&$\frac53$	&$\frac53$
\end{tabular}
\end{equation*}
To proceed, we label $\lambda_{3;3}$ as $x$ and $\lambda_{2;2}$ as $y$, at which point (iii) uniquely determines $\lambda_{3;2}$ and $\lambda_{2;1}$:
\begin{equation}
\label{equation.5 in 3 inner 2}
\begin{tabular}{p{1cm} p{1cm} p{1cm} p{1cm} p{1cm} l}
\ \,$n$&$1$&$2$&$3$&$4$&$5${\smallskip}\\ 
\hline\noalign{\smallskip}
$\lambda_{n;5}$	&	&		&				&			&$0$		{\smallskip}\\
$\lambda_{n;4}$	&	&		&				&$0$		&$0$		{\smallskip}\\
$\lambda_{n;3}$	&	&		&$x$			&$\frac23$	&$\frac53$	{\smallskip}\\
$\lambda_{n;2}$	&	&$y$	&$\frac43-x$ 	&$\frac53$	&$\frac53$	{\smallskip}\\
$\lambda_{n;1}$	&$1$&$2-y$	&$\frac53$		&$\frac53$	&$\frac53$
\end{tabular}
\end{equation}
For our particular choice of $\set{\lambda_n}_{n=1}^{5}$ and $\set{\mu_n}_{n=1}^{5}$, the above argument shows that every corresponding sequence of inner eigensteps is of the form \eqref{equation.5 in 3 inner 2}.  Conversely, one may immediately verify that any $\set{\set{\lambda_{n;m}}_{m=1}^{n}}_{n=1}^{5}$ of this form satisfies (i) and (iii) of Definition~\ref{definition.inner eigensteps} and moreover satisfies (ii) when $n=5$.  However, in order to satisfy (ii) for $n=2,3,4$, $x$ and $y$ must be chosen so that they satisfy the ten inequalities:
\begin{align}
\nonumber
\{\lambda_{3,m}\}_{m=1}^3\sqsubseteq\{\lambda_{4,m}\}_{m=1}^4&\quad\Longleftrightarrow\quad 0\leq x\leq\tfrac23\leq\tfrac43-x\leq\tfrac53,\\
\label{equation.5 in 3 inner 3}
\{\lambda_{2,m}\}_{m=1}^2\sqsubseteq\{\lambda_{3,m}\}_{m=1}^3&\quad\Longleftrightarrow\quad x\leq y\leq\tfrac43-x\leq2-y\leq\tfrac53,\\
\nonumber
\{\lambda_{1,m}\}_{m=1}^1\sqsubseteq\{\lambda_{2,m}\}_{m=1}^2&\quad\Longleftrightarrow\quad y\leq 1\leq 2-y.
\end{align}
A quick inspection reveals the system~\eqref{equation.5 in 3 inner 3} to be equivalent to the one derived in the outer eigenstep formulation \eqref{equation.5 in 3 outer 2} presented in Example~\ref{example.5 in 3 outer}, which is reducible to $0\leq x\leq\frac23$, $\max\set{\frac13,x}\leq y\leq\min\set{\frac23+x,\frac43-x}$.  Moreover, we see that the outer eigensteps~\eqref{equation.5 in 3 outer 1} that arise from $\set{\lambda_1,\lambda_2,\lambda_3}=\set{\frac53,\frac53,\frac53}$ and the inner eigensteps~\eqref{equation.5 in 3 inner 2} that arise from $\set{\lambda_1,\lambda_2,\lambda_3,\lambda_4,\lambda_5}=\set{\frac53,\frac53,\frac53,0,0}$ are but zero-padded versions of each other; the next result claims that such a result holds in general.
\end{example}

\begin{theorem}
\label{theorem.inner vs outer}
Let  $\{\lambda_n\}_{n=1}^N$ and $\{\mu_n\}_{n=1}^N$ be nonnegative and nonincreasing, and choose any $M\leq N$ such that $\lambda_n=0$ for every $n>M$.   Then every choice of outer eigensteps (Definition~\ref{definition.outer eigensteps}) corresponds to a unique choice of inner eigensteps (Definition~\ref{definition.inner eigensteps}) and vice versa, the two being zero-padded versions of each other.

Specifically, a sequence of outer eigensteps $\{\{\lambda_{n;m}\}_{m=1}^M\}_{n=0}^N$ gives rise to a sequence of inner eigensteps $\{\{\lambda_{n;m}\}_{m=1}^n\}_{n=1}^N$, where $\lambda_{n;m}:=0$ whenever $m>M$.  Conversely, a sequence of inner eigensteps $\{\{\lambda_{n;m}\}_{m=1}^n\}_{n=1}^N$ gives rise to a sequence of outer eigensteps $\{\{\lambda_{n;m}\}_{m=1}^M\}_{n=0}^N$, where $\lambda_{n;m}:=0$ whenever $m>n$.  

Moreover, $\{\lambda_{n;m}\}_{m=1}^M$ is the spectrum of the frame operator $F_n^{}F_n^*$ of $F_n=\{f_m\}_{m=1}^n$ if and only if $\{\lambda_{n;m}\}_{m=1}^n$ is the spectrum of the Gram matrix $F_n^*F_n^{}$.
\end{theorem}

The proof of Theorem~\ref{theorem.inner vs outer} is straightforward but tedious, and so we do not present it here; the interested reader can find it in~\cite{FickusMP:11}.  In the remainder of this paper, we exploit this equivalence to solve Problem~\ref{problem.main}.

%%%%%%%%%%%%%%%%%%%%%%%%%%%%%%%%%%%%%%%%%%%%%%%%%%%%%%%%%%%%%%%%
%%%%%%%%%%%%%%%%%%%%%%%%%%%%%%%%%%%%%%%%%%%%%%%%%%%%%%%%%%%%%%%%
% Section 3: Top Kill and the existence of eigensteps
%%%%%%%%%%%%%%%%%%%%%%%%%%%%%%%%%%%%%%%%%%%%%%%%%%%%%%%%%%%%%%%%
%%%%%%%%%%%%%%%%%%%%%%%%%%%%%%%%%%%%%%%%%%%%%%%%%%%%%%%%%%%%%%%%

\section{Top Kill and the existence of eigensteps}

As discussed in the previous section, the problem of constructing every Schur-Horn matrix boils down to solving Problem~\ref{problem.main}, which in light of~Theorem~\ref{theorem.necessity and sufficiency of eigensteps}, reduces to the problem of constructing every possible sequence of outer eigensteps (Definition~\ref{definition.outer eigensteps}).  Moreover, by Theorem~\ref{theorem.inner vs outer}, every sequence of outer eigensteps corresponds to a unique sequence of inner eigensteps (Definition~\ref{definition.inner eigensteps}).  We now note that if a sequence of inner eigensteps $\{\{\lambda_{n;m}\}_{m=1}^n\}_{n=1}^N$ exists, then $\set{\lambda_n}_{n=1}^{N}$ necessarily majorizes $\set{\mu_n}_{n=1}^{N}$. Indeed, letting $n=N$ in the trace property (iii) of Definition~\ref{definition.inner eigensteps} immediately gives one of the majorization conditions \eqref{equation.definition of majorization 2}; to obtain the remaining condition~\eqref{equation.definition of majorization 1} at a given $n=1,\dotsc,N-1$, note that the interlacing property (ii) gives $\lambda_{n;m}\leq\lambda_{N;m}=\lambda_m$ for all $m=1,\dotsc,n$, at which point (iii) implies
\begin{equation*}
\sum_{m=1}^{n}\mu_m
=\sum_{m=1}^{n}\lambda_{n;m}
\leq\sum_{m=1}^{n}\lambda_m.
\end{equation*}

In this section, we prove the converse result, namely that if $\set{\lambda_n}_{n=1}^{N}\succeq\set{\mu_n}_{n=1}^{N}$, then a corresponding sequence of inner eigensteps $\set{\set{\lambda_{n;m}}_{m=1}^n}_{n=1}^N$ exists.  The key idea is a new algorithm, dubbed \textit{Top Kill}, for transforming any sequence $\set{\lambda_{n;m}}_{m=1}^{n}$ that majorizes $\set{\mu_m}_{m=1}^{n}$ into a new, shorter sequence $\set{\lambda_{n;m}}_{m=1}^{n-1}$ that majorizes $\set{\mu_m}_{m=1}^{n-1}$ and also interlaces with $\set{\lambda_{n;m}}_{m=1}^{n}$.  We note that a similar idea is used to prove the Schur-Horn Theorem in~\cite{HornJ:85}; this section's contribution to the existing literature is a simple constructive proof to replace the nonconstructive, Intermediate-Value-Theorem-based existence proof of Lemma~4.3.28 of~\cite{HornJ:85}.  In the next section, these new proof techniques lead to a new result which shows how to systematically construct every valid sequence of inner eigensteps for a given $\set{\lambda_n}_{n=1}^{N}$ and $\set{\mu_n}_{n=1}^{N}$.  We now motivate Top Kill with an example:

\begin{example}
\label{example.top kill}
Let $N=3$, $\set{\lambda_1,\lambda_2,\lambda_3}=\set{\frac74,\frac34,\frac12}$ and $\set{\mu_1,\mu_2,\mu_3}=\set{1,1,1}$. Since this spectrum majorizes these lengths, we claim that there exists a corresponding sequence of inner eigensteps $\set{\set{\lambda_{n;m}}_{m=1}^n}_{n=1}^3$.  That is, recalling Definition~\ref{definition.inner eigensteps}, we claim that it is possible to find values $\set{\lambda_{1;1}}$ and $\set{\lambda_{2;1},\lambda_{2,2}}$ which satisfy the interlacing requirements (ii) that $\set{\lambda_{1;1}}\sqsubseteq\set{\lambda_{2;1},\lambda_{2,2}}\sqsubseteq\set{\frac74,\frac34,\frac12}$ as well as the trace requirements (iii) that $\lambda_{1;1}=1$ and $\lambda_{2;1}+\lambda_{2;2}=2$.  Indeed, every such sequence of eigensteps is given by the table:
\begin{equation}
\label{equation.top kill 1}
\begin{tabular}{p{1cm}p{1cm}p{1cm}l}
\ \,$n$&$1$&$2$&$3${\smallskip}\\ 
\hline\noalign{\smallskip}
$\lambda_{n;3}$	&	&		&$\frac12${\smallskip}\\
$\lambda_{n;2}$	&	&$x$	&$\frac34${\smallskip}\\
$\lambda_{n;1}$	&$1$&$2-x$	&$\frac74$\\
\end{tabular}
\end{equation}
where $x$ is required to satisfy
\begin{equation}
\label{equation.top kill 2}
\tfrac12\leq x\leq\tfrac34\leq2-x\leq\tfrac74,
\qquad
x\leq1\leq2-x.
\end{equation}
Clearly, any $x\in[\frac12,\frac34]$ will do.  However, when $N$ is large, the table analogous to~\eqref{equation.top kill 1} will contain many more variables, leading to a system of inequalities which is much larger and more complicated than~\eqref{equation.top kill 2}.  In such settings, it is not obvious how to construct even a single valid sequence of eigensteps.  As such, we consider this same simple example from a different perspective---one that leads to an eigenstep construction algorithm which is easily implementable regardless of the size of $N$.

The key idea is to view the task of constructing eigensteps as iteratively building a staircase in which the $n$th level is $\lambda_n$ units long.  For this example in particular, our goal is to build a three-step staircase where the bottom level has length $\frac74$, the second level has length $\frac34$, and the top level has length $\frac12$; the profile of such a staircase is outlined in black in each of the six subfigures of Figure~\ref{figure.top kill}.
\begin{figure}
\begin{center}
\includegraphics{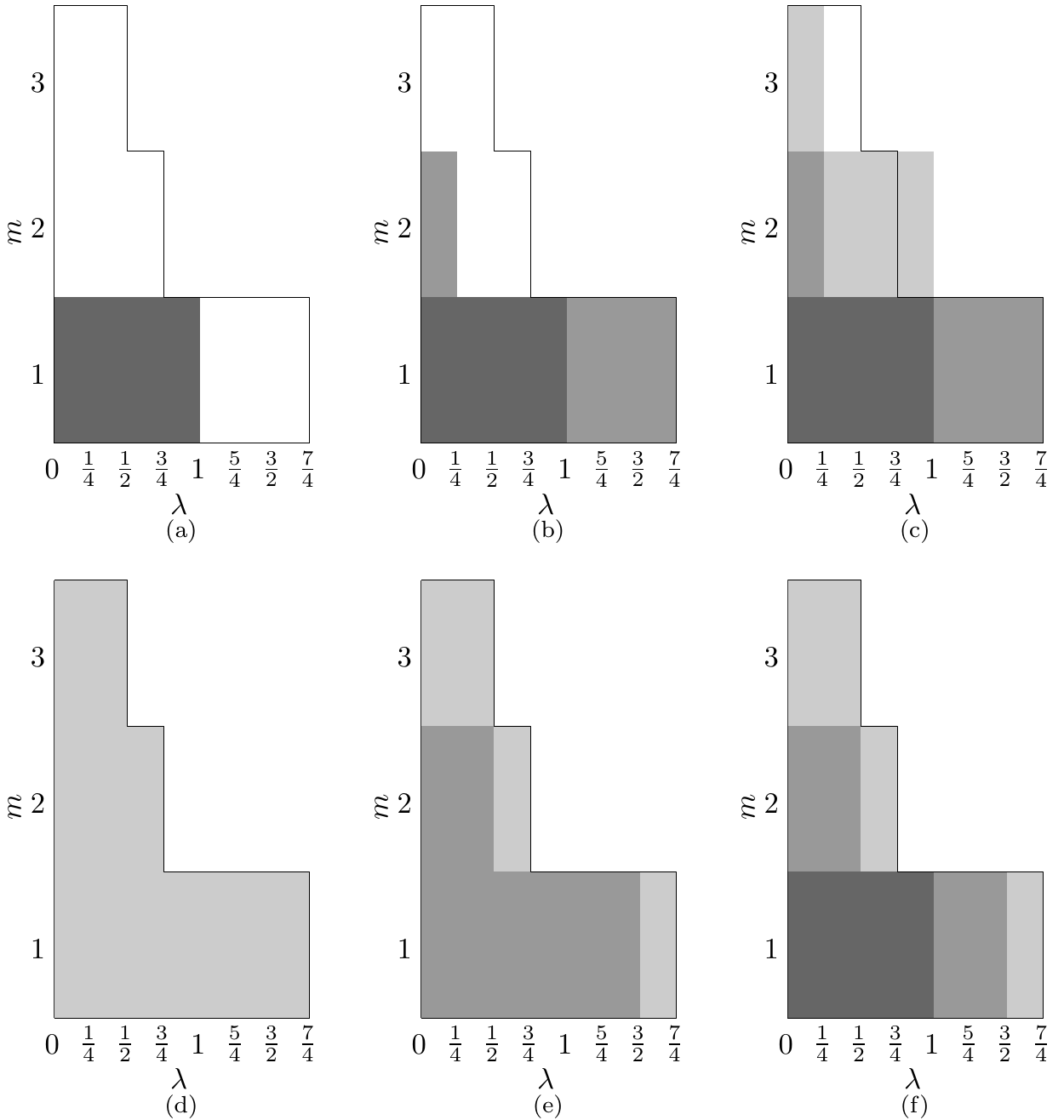}
\caption{\label{figure.top kill} Two attempts at iteratively building a sequence of inner eigensteps for $\set{\lambda_1,\lambda_2,\lambda_3}=\set{\frac74,\frac34,\frac12}$ and $\set{\mu_1,\mu_2,\mu_3}=\set{1,1,1}$.  As detailed in Example~\ref{example.top kill}, the first row represents a failed attempt in which we greedily complete the first level before focusing on those above it.  The failure arises from a lack of foresight: the second step does not build sufficient foundation for the third.  The second row represents a second attempt, one that is successful.  There, we begin with the final desired staircase and work backwards.  That is, we chip away at the three-level staircase (d) to produce a two-level one (e), and then chip away at it to produce a one-level one (f).  In each step, we remove as much as possible from the top level before turning our attention to the lower levels, subject to the interlacing constraints.  We refer to this algorithm for iteratively producing $\set{\lambda_{n-1;m}}_{m=1}^{n-1}$ from $\set{\lambda_{n;m}}_{m=1}^{n}$ as Top Kill.  Theorem~\ref{theorem.top kill} shows that Top Kill will always produce a valid sequence of eigensteps from any desired spectrum $\set{\lambda_n}_{n=1}^{N}$ that majorizes a given desired sequence of lengths $\set{\mu_n}_{n=1}^{N}$.}
\end{center}
\end{figure}
The benefit of visualizing eigensteps in this way is that the interlacing and trace conditions become intuitive staircase-building rules. Specifically, up until the $n$th step, we will have built a staircase whose levels are $\{\lambda_{n-1;m}\}_{m=1}^{n-1}$. To build on top of this staircase, we use $n$ blocks of height $1$ whose areas sum to $\mu_n$. Each of these $n$ new blocks is added to its corresponding level of the current staircase, and is required to rest entirely on top of what has been previously built. This requirement corresponds to the interlacing condition (ii) of Definition~\ref{definition.inner eigensteps}, while the trace condition (iii) corresponds to the fact that the block areas sum to $\mu_n$.

This intuition in mind, we now try to build such a staircase from the ground up.  In the first step (Figure~\ref{figure.top kill}(a)), we are required to place a single block of area $\mu_1=1$ on the first level.  The length of this first level is $\lambda_{1;1}=\mu_1$.  In the second step, we build up and out from this initial block, placing two new blocks---one on the first level and another on the second---whose total area is $\mu_2=1$.  The lengths $\lambda_{2;1}$ and $\lambda_{2;2}$ of the new first and second levels depends on how these two blocks are chosen.  In particular, choosing first and second level blocks of area $\frac34$ and $\frac14$, respectively, results in $\set{\lambda_{2;1},\lambda_{2;2}}=\set{\frac74,\frac14}$ (Figure~\ref{figure.top kill}(b)); this corresponds to a greedy pursuit of the final desired spectrum $\set{\frac74,\frac34,\frac12}$, fully completing the first level before turning our attention to the second.  The problem with this greedy approach is that it doesn't always work, as this example illustrates.  Indeed, in the third and final step, we build up and out from the staircase of Figure~\ref{figure.top kill}(b) by adding three new blocks---one each for the first, second and third levels---whose total area is $\mu_3=1$.  However, in order to maintain interlacing, the new top block must rest entirely on the existing second level, meaning that its length $\lambda_{3;3}\leq\lambda_{2;2}=\frac14$ cannot equal the desired value of $\frac12$.  That is, because of our poor choice in the second step, the ``best" we can now do is $\set{\lambda_{3;1},\lambda_{3;2},\lambda_{3;3}}=\set{\frac74,1,\frac14}$ (Figure~\ref{figure.top kill}(c)):
\begin{equation*}
\begin{tabular}{p{1cm}p{1cm}p{1cm}l}
\ \,$n$&$1$&$2$&$3${\smallskip}\\ 
\hline\noalign{\smallskip}
$\lambda_{n;3}$	&	&			&$\frac14${\smallskip}\\
$\lambda_{n;2}$	&	&$\frac14$	&$1${\smallskip}\\
$\lambda_{n;1}$	&$1$&$\frac74$	&$\frac74$\\
\end{tabular}
\end{equation*}
The reason this greedy approach fails is that it doesn't plan ahead.  Indeed, it treats the bottom levels of the staircase as the priority when, in fact, the opposite is true: the top levels are the priority since they require the most foresight.  In particular, for $\lambda_{3;3}$ to achieve its desired value of $\frac12$ in the third step, one must lay a suitable foundation in which $\lambda_{2;2}\geq\frac12$ in the second step.

In light of this realization, we make another attempt at building our staircase.  This time we begin with the final desired spectrum $\set{\lambda_{3;1},\lambda_{3;2},\lambda_{3;3}}=\set{\frac74,\frac34,\frac12}$ (Figure~\ref{figure.top kill}(d)) and work backwards.  From this perspective, our task is now to remove three blocks---the entirety of the top level, and portions of the first and second levels---whose total area is $\mu_3=1$.  Here, the interlacing requirement translates to only being permitted to remove portions of the staircase that were already exposed to the surface at the end of the previous step.  After lopping off the top level, which has area $\lambda_{3;3}=\frac12$, we need to decide how to chip away $\mu_1-\lambda_{3;3}=1-\frac12=\frac12$ units of area from the first and second levels, subject to this constraint.  At this point, we observe that in the step that follows, our first task will be to remove the remaining portion of the second level.  As such, it is to our advantage to remove as much of the second level as possible in the current step, and only then turn our attention to the lower levels.  That is, we follow Thomas Jefferson's adage, ``Never put off until tomorrow what you can do today."  We dub this approach Top Kill since it ``kills" off as much as possible from the top portions of the staircase.  For this example in particular, interlacing implies that we can at most remove a block of area $\frac14$ from the second level, leaving $\frac14$ units of area to be removed from the first; the resulting two-level staircase---the darker shade in Figure~\ref{figure.top kill}(e)---has levels of lengths $\set{\lambda_{2;1},\lambda_{2;2}}=\set{\frac32,\frac12}$.  In the second step, we then apply this same philosophy, removing the entire second level and a block of area $\mu_2-\lambda_{2;2}=1-\frac12=\frac12$ from the first, resulting in the one-level staircase (Figure~\ref{figure.top kill}(f)) in which $\set{\lambda_{1;1}}=1$.  That is, by working backwards we have produced a valid sequence of eigensteps:
\begin{equation*}
\begin{tabular}{p{1cm}p{1cm}p{1cm}l}
\ \,$n$&$1$&$2$&$3${\smallskip}\\ 
\hline\noalign{\smallskip}
$\lambda_{n;3}$	&	&			&$\frac14${\smallskip}\\
$\lambda_{n;2}$	&	&$\frac12$	&$1${\smallskip}\\
$\lambda_{n;1}$	&$1$&$\frac32$	&$\frac74$\\
\end{tabular}
\end{equation*}
\end{example}

The preceding example illustrated a systematic ``Top Kill" approach for building eigensteps; we now express these ideas more rigorously.  As can be seen in the bottom row of Figure~\ref{figure.top kill}, Top Kill generally picks $\lambda_{n-1;m}:=\lambda_{n;m+1}$ for the larger $m$'s.  Top Kill also picks $\lambda_{n-1;m}:=\lambda_{n;m}$ for the smaller $m$'s. The level that separates the larger $m$'s from the smaller $m$'s is the lowest level from which a nontrivial area is removed.  For this level, say level $k$, we have $\lambda_{n;k+1}<\mu_n\leq\lambda_{n;k}$.  In the levels above $k$, we have already removed a total of $\lambda_{n;k+1}$ units of area, leaving $\mu_n-\lambda_{n;k+1}$ to be chipped away from $\lambda_{n;k}$, yielding $\lambda_{n-1;k}:=\lambda_{n;k}-(\mu_n-\lambda_{n;k+1})$.  The following theorem confirms that Top Kill always produces eigensteps whenever it is possible to do so:

\begin{theorem} \label{theorem.top kill}
Suppose $\{\lambda_{n;m}\}_{m=1}^{n}\succeq\{\mu_{m}\}_{m=1}^{n}$, and define $\{\lambda_{n-1;m}\}_{m=1}^{n-1}$ according to Top Kill, that is, pick any $k$ such that $\lambda_{n;k+1}\leq\mu_{n}\leq\lambda_{n;k}$, and for each $m=1,\ldots,n-1$, define:
\begin{equation} \label{eq.top kill}
\lambda_{n-1;m}
:=\left\{\begin{array}{ll}\lambda_{n;m},&1\leq m\leq k-1,\\
\lambda_{n;k}+\lambda_{n;k+1}-\mu_{n},&m=k,\\
\lambda_{n;m+1},&k+1\leq m\leq n-1.
\end{array}\right.
\end{equation}
Then $\{\lambda_{n-1;m}\}_{m=1}^{n-1}\sqsubseteq\{\lambda_{n;m}\}_{m=1}^{n}$ and $\{\lambda_{n-1;m}\}_{m=1}^{n-1}\succeq\{\mu_m\}_{m=1}^{n-1}$.

Furthermore, given any nonnegative nonincreasing sequences $\{\lambda_{n}\}_{n=1}^N$ and $\{\mu_{n}\}_{n=1}^N$ such that $\{\lambda_{n}\}_{n=1}^N\succeq\{\mu_{n}\}_{n=1}^N$, define $\lambda_{N;m}:=\lambda_m$ for every $m=1,\ldots,N$, and for each $n=N,\ldots,2$, consecutively define $\{\lambda_{n-1;m}\}_{m=1}^{n-1}$ according to Top Kill. Then $\{\{\lambda_{n;m}\}_{m=1}^n\}_{n=1}^N$ are inner eigensteps.
\end{theorem}

\begin{proof}
For the sake of notational simplicity, we denote $\{\alpha_m\}_{m=1}^{n-1}:=\{\lambda_{n-1;m}\}_{m=1}^{n-1}$ and $\{\beta_{m}\}_{m=1}^{n}:=\{\lambda_{n;m}\}_{m=1}^{n}$.  Since $\set{\beta_m}_{m=1}^{n}\succeq\set{\mu_m}_{m=1}^{n}$, we necessarily have that $\beta_n\leq\mu_n\leq\mu_1\leq\beta_1$, and so there exists $k=1,\dotsc,n-1$ such that $\beta_{k+1}\leq\mu_{n}\leq\beta_k$.  Though this $k$ may not be unique when subsequent $\beta_m$'s are equal, a quick inspection reveals that any appropriate choice of $k$ will yield the same $\alpha_{m}$'s, and so Top Kill is well-defined.  To prove $\{\alpha_m\}_{m=1}^{n-1}\sqsubseteq\{\beta_m\}_{m=1}^{n}$, we need to show that
\begin{equation}
\label{eq.interlacing nts}
\beta_{m+1}\leq\alpha_m\leq\beta_m
\end{equation}
for every $m=1,\ldots,n-1$. If $1\leq m\leq k-1$, then $\alpha_m:=\beta_m$, and so the right-hand inequality of \eqref{eq.interlacing nts} holds with equality, at which point the left-hand inequality is immediate.  Similarly, if $k+1\leq m\leq n-1$, then $\alpha_m:=\beta_{m+1}$, and so \eqref{eq.interlacing nts} holds with equality on the left-hand side. Lastly if $m=k$, then $\alpha_k:=\beta_k+\beta_{k+1}-\mu_{n}$, and our assumption that $\beta_{k+1}\leq\mu_{n}\leq\beta_k$ gives \eqref{eq.interlacing nts} in this case:
\begin{equation*}
\beta_{k+1}\leq\beta_k+\beta_{k+1}-\mu_{n}\leq\beta_k.
\end{equation*}
Thus, Top Kill produces $\{\alpha_m\}_{m=1}^{n-1}$ such that $\{\alpha_m\}_{m=1}^{n-1}\sqsubseteq\{\beta_m\}_{m=1}^{n}$.  We next show that  $\{\alpha_m\}_{m=1}^{n-1}\succeq\{\mu_m\}_{m=1}^{n-1}$.
If $j\leq k-1$, then since $\{\beta_m\}_{m=1}^{n}\succeq\{\mu_m\}_{m=1}^{n}$, we have
\begin{equation*}
\sum_{m=1}^j\alpha_m
=\sum_{m=1}^j\beta_m
\geq\sum_{m=1}^j\mu_m.
\end{equation*}
On the other hand, if $j\geq k$, we have
\begin{equation}
\label{eq.majorization proof 1}
\sum_{m=1}^j\alpha_m
=\sum_{m=1}^{k-1}\beta_m+(\beta_k+\beta_{k+1}-\mu_{n})+\sum_{m=k+1}^j\beta_{m+1}
=\sum_{m=1}^{j+1}\beta_m-\mu_{n},
\end{equation}
with the understanding that a sum over an empty set of indices is zero. We continue \eqref{eq.majorization proof 1} by using the facts that $\{\beta_m\}_{m=1}^{n}\succeq\{\mu_m\}_{m=1}^{n}$ and $\mu_{j+1}\geq\mu_{n}$:
\begin{equation}
\label{eq.majorization proof 2}
\sum_{m=1}^j\alpha_m
=\sum_{m=1}^{j+1}\beta_m-\mu_{n}
\geq\sum_{m=1}^{j+1}\mu_m-\mu_{n}
\geq\sum_{m=1}^{j}\mu_{m}.
\end{equation}
Note that when $j=n$, the inequalities in \eqref{eq.majorization proof 2} become equalities, giving the final trace condition.

For the final conclusion, note that one application of Top Kill transforms a sequence $\set{\lambda_{n;m}}_{m=1}^{n}$ that majorizes $\set{\mu_m}_{m=1}^{n}$ into a shorter sequence $\set{\lambda_{n-1;m}}_{m=1}^{n-1}$ that interlaces with $\set{\lambda_{n;m}}_{m=1}^{n}$ and majorizes $\set{\mu_m}_{m=1}^{n-1}$.  As such, one may indeed start with $\lambda_{N;m}:=\lambda_m$ and apply Top Kill $N-1$ times to produce a sequence $\set{\set{\lambda_{n;m}}_{m=1}^{n}}_{n=1}^{N}$ that immediately satisfies Definition~\ref{definition.inner eigensteps}.
\qquad
\end{proof}

%%%%%%%%%%%%%%%%%%%%%%%%%%%%%%%%%%%%%%%%%%%%%%%%%%%%%%%%%%%%%%%%
%%%%%%%%%%%%%%%%%%%%%%%%%%%%%%%%%%%%%%%%%%%%%%%%%%%%%%%%%%%%%%%%
% Section 4: Parametrizing eigensteps
%%%%%%%%%%%%%%%%%%%%%%%%%%%%%%%%%%%%%%%%%%%%%%%%%%%%%%%%%%%%%%%%
%%%%%%%%%%%%%%%%%%%%%%%%%%%%%%%%%%%%%%%%%%%%%%%%%%%%%%%%%%%%%%%%

\section{Parametrizing eigensteps}

In the previous section, we discussed Top Kill, an algorithm designed to construct a sequence of inner eigensteps from given nonnegative nonincreasing sequences $\set{\lambda_n}_{n=1}^{N}$ and $\set{\mu_n}_{n=1}^{N}$.  In this section, we use the intuition underlying Top Kill to find a systematic method for producing all such eigensteps.  To be precise, treating the values $\set{\set{\lambda_{n;m}}_{m=1}^{n}}_{n=1}^{N-1}$ as independent variables, it is not difficult to show that the set of all inner eigensteps for a given $\set{\lambda_n}_{n=1}^{N}$ and $\set{\mu_n}_{n=1}^{N}$ form a convex polytope in $\bbR^{N(N-1)/2}$.  Our goal is to find a useful, implementable parametrization of this polytope.

We begin by noting that this polytope is nonempty precisely when $\set{\lambda_n}_{n=1}^{N}$ majorizes $\set{\mu_n}_{n=1}^{N}$.  Indeed, as noted at the beginning of the previous section, if such a sequence of eigensteps exists, then we necessarily have that $\set{\lambda_n}_{n=1}^{N}\succeq\set{\mu_n}_{n=1}^{N}$.  Conversely, if $\set{\lambda_n}_{n=1}^{N}\succeq\set{\mu_n}_{n=1}^{N}$, then Theorem~\ref{theorem.top kill} states that Top Kill will produce a valid sequence of eigensteps from $\set{\lambda_n}_{n=1}^{N}$ and $\set{\mu_n}_{n=1}^{N}$.  Note this implies that for a given $\set{\lambda_n}_{n=1}^{N}$ and $\set{\mu_n}_{n=1}^{N}$, if any given strategy for building eigensteps is successful, then Top Kill will also succeed.  In this sense, Top Kill is an optimal strategy.  However, Top Kill alone will not suffice to parametrize our polytope, since for a given feasible $\set{\lambda_n}_{n=1}^{N}$ and $\set{\mu_n}_{n=1}^{N}$, it only produces a single sequence of eigensteps when, in fact, there may be infinitely many such sequences.  In the work that follows, we view these non-Top-Kill-produced eigensteps as the result of applying suboptimal generalizations of Top Kill to $\set{\lambda_n}_{n=1}^{N}$ and $\set{\mu_n}_{n=1}^{N}$.

For example, if $\set{\lambda_1,\lambda_2,\lambda_3,\lambda_4,\lambda_5}=\set{\frac53,\frac53,\frac53,0,0}$ and $\mu_n=1$ for all $n=1,\dotsc,5$, every sequence of inner eigensteps corresponds to a choice of the unknown values in~\eqref{equation.5 in 3 inner 0} which satisfies the interlacing and trace conditions (ii) and (iii) of Definition~\ref{definition.inner eigensteps}.  There are $10$ unknows in~\eqref{equation.5 in 3 inner 0}, and the set of all such eigensteps is a convex polytope in $\bbR^{10}$.  Though this dimension can be reduced by exploiting the interlacing and trace conditions---the $10$ unknowns in~\eqref{equation.5 in 3 inner 0} can be reduced to the two unknowns in~\eqref{equation.5 in 3 inner 2}---this approach to constructing all eigensteps nevertheless requires one to simplify large systems of coupled inequalities, such as~\eqref{equation.5 in 3 inner 3}.

We suggest a different method for parametrizing this polytope: to systematically pick the values $\set{\set{\lambda_{n;m}}_{m=1}^{n}}_{n=1}^{4}$ one at a time.  Top Kill is one way to do this: working from the top levels down, we chip away $\mu_5=1$ units of area from $\set{\lambda_{5;m}}_{m=1}^{5}$ to successively produce $\lambda_{4;4}=0$, $\lambda_{4;3}=\frac23$, $\lambda_{4;2}=\frac53$ and $\lambda_{4;1}=\frac53$; we then repeat this process to transform $\set{\lambda_{4;m}}_{m=1}^{4}$ into $\set{\lambda_{3;m}}_{m=1}^{3}$, and so on; the specific values can be obtained by letting $(x,y)=(0,\frac13)$ in~\eqref{equation.5 in 3 inner 2}.  We seek to generalize Top Kill to find \textit{all} ways of picking the $\lambda_{n;m}$'s one at a time.  As in Top Kill, we work backwards: we first find all possibilities for $\lambda_{4;4}$, then the possibilities for $\lambda_{4;3}$ in terms of our choice of $\lambda_{4;4}$, then the possibilities for $\lambda_{4;2}$ in terms of our choices of $\lambda_{4;4}$ and $\lambda_{4;3}$, and so on.  That is, we iteratively parametrize our convex polytope in the following order:
\begin{equation*}
\lambda_{4;4},\quad
\lambda_{4;3},\quad 
\lambda_{4;2},\quad
\lambda_{4;1},\quad 
\lambda_{3;3},\quad
\lambda_{3;2},\quad 
\lambda_{3;1},\quad 
\lambda_{2;2},\quad
\lambda_{2;1},\quad
\lambda_{1;1}.
\end{equation*}

More generally, for any $\set{\lambda_n}_{n=1}^{N}$ and $\set{\mu_n}_{n=1}^{N}$ such that $\set{\lambda_n}_{n=1}^{N}\succeq\set{\mu_n}_{n=1}^{N}$ we construct every possible sequence of eigensteps $\set{\set{\lambda_{n;m}}_{m=1}^{n}}_{n=1}^{N}$ by finding all possibilities for any given $\lambda_{n-1;k}$ in terms of $\lambda_{n';m}$ where either $n'>n-1$ or $n'=n-1$ and $m>k$.  Certainly, any permissible choice for $\lambda_{n-1;k}$ must satisfy the interlacing criteria (ii) of Definition~\ref{definition.inner eigensteps}, and so we have bounds $\lambda_{n;k+1}\leq\lambda_{n-1;k}\leq\lambda_{n;k}$. Other necessary bounds arise from the majorization conditions. Indeed, in order to have both $\{\lambda_{n;m}\}_{m=1}^{n}\succeq\{\mu_{m}\}_{m=1}^{n}$ and $\{\lambda_{n-1;m}\}_{m=1}^{n-1}\succeq\{\mu_m\}_{m=1}^{n-1}$ we need
\begin{equation}
\label{eq.intuition 1}
\mu_n
=\sum_{m=1}^n\mu_m-\sum_{m=1}^{n-1}\mu_m
=\sum_{m=1}^n\lambda_{n;m}-\sum_{m=1}^{n-1}\lambda_{n-1;m},
\end{equation}
and so we may view $\mu_n$ as the total change between the eigenstep spectra.  Having already selected $\lambda_{n-1;n-1},\dots,\lambda_{n-1;k+1}$, we've already imposed a certain amount of change between the spectra, and so we are limited in how much we can change the $k$th eigenvalue. Continuing \eqref{eq.intuition 1}, this fact can be expressed as
\begin{equation}
\label{eq.intuition 2}
\mu_n
=\lambda_{n;n}+\sum_{m=1}^{n-1}(\lambda_{n;m}-\lambda_{n-1;m})
\geq\lambda_{n;n}+\sum_{m=k}^{n-1}(\lambda_{n;m}-\lambda_{n-1;m}),
\end{equation}
where the inequality follows from the fact that the summands $\lambda_{n;m}-\lambda_{n-1;m}$ are nonnegative if $\{\lambda_{n-1;m}\}_{m=1}^{n-1}$ is to be chosen so that $\{\lambda_{n-1;m}\}_{m=1}^{n-1}\sqsubseteq\{\lambda_{n;m}\}_{m=1}^{n}$. Rearranging \eqref{eq.intuition 2} then gives a second lower bound on $\lambda_{n-1;k}$ to go along with our previously mentioned requirement that $\lambda_{n-1;k}\geq\lambda_{n;k+1}$:
\begin{equation}
\label{eq.intuition 3}
\lambda_{n-1;k}\geq\sum_{m=k}^{n}\lambda_{n;m}-\sum_{m=k+1}^{n-1}\lambda_{n-1;m}-\mu_n.
\end{equation}

We next apply the intuition behind Top Kill to obtain other upper bounds on $\lambda_{n-1;k}$ to go along with our previously mentioned requirement that $\lambda_{n-1;k}\leq\lambda_{n;k}$.  We caution that what follows is not a rigorous argument for the remaining upper bound on $\lambda_{n-1;k}$, but rather an informal derivation of this bound's expression; the legitimacy of this derivation is formally confirmed in the proof of Theorem~\ref{theorem.parametrize eigensteps}.  Recall that at this point in the narrative, we have already selected $\set{\lambda_{n-1;m}}_{m=k+1}^{n-1}$ and are attempting to find all possible choices $\lambda_{n-1;k}$ that will allow the remaining values $\set{\lambda_{n-1;m}}_{m=1}^{k-1}$ to be chosen in such a way that:
\begin{equation}
\label{eq.intuition 3.5}
\set{\lambda_{n-1;m}}_{m=1}^{n-1}\sqsubseteq\set{\lambda_{n;m}}_{m=1}^{n},
\qquad 
\set{\lambda_{n-1;m}}_{m=1}^{n-1}\succeq\set{\mu_{m}}_{m=1}^{n-1}.
\end{equation}
To do this, we recall our staircase-building intuition from the previous section: if it is possible to build a given staircase, then one way to do this is to assign maximal priority to the highest levels, as these are the most difficult to build.  As such, for a given choice of $\lambda_{n-1;k}$, if it is possible to choose $\set{\lambda_{n-1;m}}_{m=1}^{k-1}$ in such a way that \eqref{eq.intuition 3.5} holds, then it is reasonable to expect that one way of doing this is to pick $\lambda_{n-1;k-1}$ by chipping away as much as possible from $\lambda_{n;k-1}$, then pick $\lambda_{n-1;k-2}$ by chipping away as much as possible from $\lambda_{n;k-2}$, and so on.  That is, we pick some arbitrary value $\lambda_{n-1;k}$, and to test its legitimacy, we apply the Top Kill algorithm to construct the remaining undetermined values $\set{\lambda_{n-1;m}}_{m=1}^{k-1}$; we then check whether or not $\set{\lambda_{n-1;m}}_{m=1}^{n-1}\succeq\set{\mu_{m}}_{m=1}^{n-1}$.

To be precise, note that prior to applying Top Kill, the remaining spectrum is $\{\lambda_{n;m}\}_{m=1}^{k-1}$, and that the total amount we will chip away from this spectrum is
\begin{equation}
\label{eq.intuition 4}
\mu_n-\bigg(\lambda_{n;n}+\sum_{m=k}^{n-1}(\lambda_{n;m}-\lambda_{n-1;m})\bigg).
\end{equation}
To ensure that our choice of $\lambda_{n-1;k-1}$ satisfies $\lambda_{n-1;k-1}\geq\lambda_{n;k}$, we artificially reintroduce $\lambda_{n;k}$ to both~\eqref{eq.intuition 4} and the remaining spectrum $\set{\lambda_{n;m}}_{m=1}^{k-1}$ before applying Top Kill.  That is, we apply Top Kill to $\{\beta_m\}_{m=1}^n:=\{\lambda_{n;m}\}_{m=1}^{k}\cup\{0\}_{m=k+1}^n$, where
\begin{equation}
\label{eq.intuition 5}
\mu
:=\mu_n-\bigg(\lambda_{n;n}+\sum_{m=k}^{n-1}(\lambda_{n;m}-\lambda_{n-1;m})\bigg)+\lambda_{n;k}
=\mu_n-\sum_{m=k+1}^{n}\lambda_{n;m}+\sum_{m=k}^{n-1}\lambda_{n-1;m}.
\end{equation}
Specifically in light of Theorem~\ref{theorem.top kill}, in order to optimally subtract $\mu$ units of area from $\{\beta_m\}_{m=1}^n$, we first pick $j$ such that $\beta_{j+1}\leq\mu\leq\beta_j$. We then use \eqref{eq.top kill} to produce a zero-padded version of the remaining new spectrum $\{\lambda_{n-1;m}\}_{m=1}^{k-1}\cup\{0\}_{m=k}^n$:
\begin{equation*}
\lambda_{n-1;m}=\left\{\begin{array}{ll}\lambda_{n;m},&1\leq m\leq j-1,\\\displaystyle{\lambda_{n;j}+\lambda_{n;j+1}-\mu_n+\sum_{m'=k+1}^n\lambda_{n;m'}-\sum_{m'=k}^{n-1}\lambda_{n-1;m'},}&m=j\\\lambda_{n;m+1},&j+1\leq m\leq k-1.\end{array}\right.
\end{equation*}
Picking ${l}$ such that $j+1\leq{l}\leq k$, we now sum the above values of $\lambda_{n-1;m}$ to obtain
\begin{align}
\nonumber
\sum_{m=1}^{{l}-1}\lambda_{n-1;m}
&=\sum_{m=1}^{j-1}\lambda_{n-1;m}+\lambda_{n-1;j}+\sum_{m=j+1}^{{l}-1}\lambda_{n-1;m}\\
\label{eq.intuition 5.5}
&=\sum_{m=1}^{l}\lambda_{n;m}-\mu_n+\sum_{m=k+1}^n\lambda_{n;m}-\sum_{m=k}^{n-1}\lambda_{n-1;m}.
\end{align}
Adding $\displaystyle\sum_{m=1}^n\mu_m-\sum_{m=1}^n\lambda_{n;m}=0$ to the right-hand side of \eqref{eq.intuition 5.5} then yields
\begin{align}
\nonumber
\sum_{m=1}^{{l}-1}\lambda_{n-1;m}
&=\sum_{m=1}^{l}\lambda_{n;m}-\mu_n+\sum_{m=k+1}^n\lambda_{n;m}-\sum_{m=k}^{n-1}\lambda_{n-1;m}+\sum_{m=1}^n\mu_m-\sum_{m=1}^n\lambda_{n;m}\\
\label{eq.intuition 6}
&=\sum_{m=1}^{n-1}\mu_m-\sum_{m={l}+1}^k\lambda_{n;m}-\sum_{m=k}^{n-1}\lambda_{n-1;m}.
\end{align}
Now, in order for $\{\lambda_{n-1;m}\}_{m=1}^{n-1}\succeq\{\mu_m\}_{m=1}^{n-1}$ as desired, \eqref{eq.intuition 6} must satisfy
\begin{equation}
\label{eq.intuition 6.5}
\sum_{m=1}^{{l}-1}\mu_m
\leq\sum_{m=1}^{{l}-1}\lambda_{n-1;m}
=\sum_{m=1}^{n-1}\mu_m-\sum_{m={l}+1}^k\lambda_{n;m}-\sum_{m=k}^{n-1}\lambda_{n-1;m}.
\end{equation}
Solving for $\lambda_{n-1;k}$ in \eqref{eq.intuition 6.5} then gives
\begin{equation}
\label{eq.intuition 7}
\lambda_{n-1;k}\leq\sum_{m={l}}^{n-1}\mu_m-\sum_{m={l}+1}^k\lambda_{n;m}-\sum_{m=k+1}^{n-1}\lambda_{n-1;m}.
\end{equation}
Note that, according to how we derived it, \eqref{eq.intuition 7} is valid when $j+1\leq{l}\leq k$. As established in the following theorem, this bound actually holds when ${l}=1,\ldots,k$. Overall, the interlacing conditions, \eqref{eq.intuition 3}, and \eqref{eq.intuition 7} are precisely the bounds that we verify in the following result:

\begin{theorem}
\label{theorem.parametrize eigensteps}
Suppose $\{\lambda_{n;m}\}_{m=1}^{n} \succeq \{\mu_{m}\}_{m=1}^{n}$.   Then $\{\lambda_{n-1;m}\}_{m=1}^{n-1} \succeq \{\mu_{m}\}_{m=1}^{n-1}$ and $\{\lambda_{n-1;m}\}_{m=1}^{n-1}\sqsubseteq\{\lambda_{n;m}\}_{m=1}^{n}$ if and only if $\lambda_{n-1;k}\in[A_{n-1;k},B_{n-1;k}]$ for every $k=1,\ldots,n-1$, where
\begin{align}
\label{eq.lower bound}
A_{n-1;k}&:=\max \bigg\{ \lambda_{n;k+1}, \sum_{m=k}^{n} \lambda_{n;m} - \sum_{m=k+1}^{n-1} \lambda_{n-1;m} - \mu_n \bigg\},\\
\label{eq.upper bound}
B_{n-1;k}&:=\min \bigg\{ \lambda_{n;k}, \min_{{l}=1,\dots,k} \bigg\{\sum_{m={l}}^{n-1} \mu_m - \sum_{m={l}+1}^{k} \lambda_{n;m} - \sum_{m=k+1}^{n-1} \lambda_{n-1;m} \bigg\}\bigg\}.
\end{align}
Here, we use the convention that sums over empty sets of indices are zero.  Moreover, suppose $\lambda_{n-1;n-1},\dots,\lambda_{n-1;k+1}$ are consecutively chosen to satisy these bounds.
Then $A_{n-1;k}\leq B_{n-1;k}$, and so $\lambda_{n-1;k}$ can also be chosen from such an interval.
\end{theorem}
\begin{proof}
For the sake of notational simplicity, we let $\{ \alpha_{m} \}_{m=1}^{n-1}:=\{ \lambda_{n-1;m} \}_{m=1}^{n-1}$, $\{ \beta_{m} \}_{m=1}^{n}:=\{ \lambda_{n;m} \}_{m=1}^{n}$, $A_k:=A_{n-1;k}$, and $B_k:=B_{n-1;k}$.

($\Rightarrow$)  Suppose $\{\alpha_m\}_{m=1}^{n-1} \succeq \{\mu_{m}\}_{m=1}^{n-1}$ and $\{\alpha_m\}_{m=1}^{n-1}\sqsubseteq\{\beta_m\}_{m=1}^{n}$.  Fix any particular $k=1,\ldots,n-1$. Note that interlacing gives $\beta_{k+1}\leq\alpha_k\leq\beta_k$, which accounts for the first entries in \eqref{eq.lower bound} and \eqref{eq.upper bound}. We first show $\alpha_k\geq A_k$. Since $\{\beta_m\}_{m=1}^n\succeq\{\mu_m\}_{m=1}^n$ and $\{\alpha_m\}_{m=1}^{n-1} \succeq \{\mu_{m}\}_{m=1}^{n-1}$, then
\begin{equation}
\label{eq.pes 1}
\mu_n
=\sum_{m=1}^n\mu_m-\sum_{m=1}^{n-1}\mu_m
=\sum_{m=1}^n\beta_m-\sum_{m=1}^{n-1}\alpha_m
=\beta_n+\sum_{m=1}^{n-1}(\beta_m-\alpha_m).
\end{equation}
Since $\{\alpha_m\}_{m=1}^{n-1}\sqsubseteq\{\beta_m\}_{m=1}^{n}$, the summands in \eqref{eq.pes 1} are nonnegative, and so
\begin{equation}
\label{eq.pes 2}
\mu_n
\geq \beta_n+\sum_{m=k}^{n-1}(\beta_m-\alpha_m)
=\sum_{m=k}^n\beta_m-\sum_{m=k+1}^{n-1}\alpha_m-\alpha_k.
\end{equation}
Isolating $\alpha_k$ in \eqref{eq.pes 2} and combining with the fact that $\alpha_k\geq\beta_{k+1}$ gives $\alpha_k\geq A_k$. We next show that $\alpha_k\leq B_k$. Fix ${l}=1,\ldots,k$. Then $\{\alpha_m\}_{m=1}^{n-1} \succeq \{\mu_{m}\}_{m=1}^{n-1}$ implies $\sum_{m=1}^{{l}-1}\alpha_m\geq\sum_{m=1}^{{l}-1}\mu_m$ and $\sum_{m=1}^{n-1}\alpha_m=\sum_{m=1}^{n-1}\mu_m$, and so subtracting gives
\begin{equation}
\label{eq.pes 3}
\sum_{m={l}}^{n-1}\mu_m
\geq\sum_{m={l}}^{n-1}\alpha_m
=\sum_{m=k}^{n-1}\alpha_m+\sum_{m={l}}^{k-1}\alpha_m
\geq\sum_{m=k}^{n-1}\alpha_m+\sum_{m={l}}^{k-1}\beta_{m+1},
\end{equation}
where the second inequality follows from $\{\alpha_m\}_{m=1}^{n-1}\sqsubseteq\{\beta_m\}_{m=1}^{n}$. Since our choice for ${l}=1,\ldots,k$ was arbitrary, isolating $\alpha_k$ in \eqref{eq.pes 3} and combining with the fact that $\alpha_k\leq\beta_k$ gives $\alpha_k\leq B_k$.

($\Leftarrow$)
Now suppose $A_k\leq\alpha_k\leq B_k$ for every $k=1,\ldots,n-1$. Then the first entries in \eqref{eq.lower bound} and \eqref{eq.upper bound} give $\beta_{k+1}\leq\alpha_k\leq\beta_k$ for every $k=1,\ldots,n-1$, that is, $\{\alpha_m\}_{m=1}^{n-1}\sqsubseteq\{\beta_m\}_{m=1}^{n}$. It remains to be shown that $\{\alpha_m\}_{m=1}^{n-1} \succeq \{\mu_{m}\}_{m=1}^{n-1}$. Since $\alpha_k\leq B_k$ for every $k=1,\ldots,n-1$, then
\begin{equation}
\label{eq.pes 4}
\alpha_k
\leq\sum_{m={l}}^{n-1}\mu_m-\sum_{m={l}+1}^k\beta_m-\sum_{m=k+1}^{n-1}\alpha_m
\qquad\forall k=1,\ldots,n-1,~~{l}=1,\ldots,k.
\end{equation}
Rearranging \eqref{eq.pes 4} in the case where ${l}=k$ gives
\begin{equation}
\label{eq.pes 5}
\sum_{m=k}^{n-1}\alpha_m
\leq\sum_{m=k}^{n-1}\mu_m
\qquad\forall k=1,\ldots,n-1.
\end{equation}
Moreover, $\alpha_1\geq A_1$ implies $\alpha_1\geq\sum_{m=1}^n\beta_m-\sum_{m=2}^{n-1}\alpha_m-\mu_n$. Rearranging this inequality and applying $\{\beta_m\}_{m=1}^{n} \succeq \{\mu_{m}\}_{m=1}^{n}$ then gives
\begin{equation}
\label{eq.pes 7}
\sum_{m=1}^{n-1}\alpha_m
\geq\sum_{m=1}^n\beta_m-\mu_n
=\sum_{m=1}^{n-1}\mu_m.
\end{equation}
Combining \eqref{eq.pes 7} with \eqref{eq.pes 5} in the case where $k=1$ gives
\begin{equation}
\label{eq.pes 8}
\sum_{m=1}^{n-1}\alpha_m
=\sum_{m=1}^{n-1}\mu_m.
\end{equation}
Subtracting \eqref{eq.pes 5} from \eqref{eq.pes 8} completes the proof that $\{\alpha_m\}_{m=1}^{n-1} \succeq \{\mu_{m}\}_{m=1}^{n-1}$.

For the final claim, we first show that the claim holds for $k=n-1$, namely that $A_{n-1}\leq B_{n-1}$. Explicitly, we need to show that
\begin{equation}
\label{eq.pes 9}
\max\{\beta_n,\beta_{n-1}+\beta_n-\mu_n\}
\leq\min\bigg\{\beta_{n-1},\min_{{l}=1,\ldots,n-1}\bigg\{\sum_{m={l}}^{n-1}\mu_m-\sum_{m={l}+1}^{n-1}\beta_m\bigg\}\bigg\}.
\end{equation}
Note that \eqref{eq.pes 9} is equivalent to the following inequalities holding simultaneously:\bigskip
\begin{itemize}
\item[(i)] $\beta_n\leq\beta_{n-1}$,\bigskip
\item[(ii)] $\beta_{n-1}+\beta_n-\mu_n\leq\beta_{n-1}$,\medskip
\item[(iii)] $\displaystyle\beta_n\leq\sum_{m={l}}^{n-1}\mu_m-\sum_{m={l}+1}^{n-1}\beta_m\quad\forall{l}=1,\ldots,n-1$,
\item[(iv)] $\displaystyle\beta_{n-1}+\beta_n-\mu_n\leq\sum_{m={l}}^{n-1}\mu_m-\sum_{m={l}+1}^{n-1}\beta_m\quad\forall{l}=1,\ldots,n-1$.
\end{itemize}
First, (i) follows immediately from the fact that $\{\beta_m\}_{m=1}^n$ is nonincreasing. Next, rearranging (ii) gives $\beta_n\leq\mu_n$, which follows from $\{\beta_m\}_{m=1}^{n} \succeq \{\mu_{m}\}_{m=1}^{n}$. For (iii), the facts that $\{\beta_m\}_{m=1}^{n} \succeq \{\mu_{m}\}_{m=1}^{n}$ and $\{\mu_m\}_{m=1}^n$ is nonincreasing imply
\begin{equation*}
\sum_{m={l}+1}^n\beta_m\leq\sum_{m={l}+1}^n\mu_m\leq\sum_{m={l}}^{n-1}\mu_m\qquad \forall{l}=1,\ldots,n-1,
\end{equation*}
which in turn implies (iii).
Also for (iv), the facts that $\{\beta_m\}_{m=1}^n$ is nonincreasing and $\{\beta_m\}_{m=1}^{n} \succeq \{\mu_{m}\}_{m=1}^{n}$ imply
\begin{equation*}
\beta_{n-1}+\sum_{m={l}+1}^n\beta_m\leq\sum_{m={l}}^n\beta_m\leq\sum_{m={l}}^n\mu_m \qquad \forall{l}=1,\ldots,n-1,
\end{equation*}
which in turn implies (iv). We now proceed by induction. Assume $\alpha_{k+1}$ satisfies $A_{k+1}\leq\alpha_{k+1}\leq B_{k+1}$. Given this assumption, we need to show that $A_{k}\leq B_{k}$. Considering the definitions \eqref{eq.lower bound} and \eqref{eq.upper bound} of $A_k$ and $B_k$, this is equivalent to the following inequalities holding simultaneously:\medskip
\begin{itemize}
\item[(i)] $\beta_{k+1}\leq\beta_k$,\medskip
\item[(ii)] $\displaystyle\sum_{m=k}^n\beta_m-\sum_{m=k+1}^{n-1}\alpha_m-\mu_n\leq\beta_k$,
\item[(iii)] $\displaystyle\beta_{k+1}\leq\sum_{m={l}}^{n-1}\mu_m-\sum_{m={l}+1}^k\beta_m-\sum_{m=k+1}^{n-1}\alpha_m\quad \forall{l}=1,\ldots,k$,
\item[(iv)] $\displaystyle\sum_{m=k}^n\beta_m-\sum_{m=k+1}^{n-1}\alpha_m-\mu_n\leq\sum_{m={l}}^{n-1}\mu_m-\sum_{m={l}+1}^k\beta_m-\sum_{m=k+1}^{n-1}\alpha_m\quad\forall{l}=1,\ldots,k.$ \\
\end{itemize}
Again, the fact that $\{\beta_m\}_{m=1}^n$ is nonincreasing implies (i). Next, $\alpha_{k+1}\geq A_{k+1}$ gives
\begin{equation*}
\alpha_{k+1}\geq\sum_{m=k+1}^n\beta_m-\sum_{m=k+2}^{n-1}\alpha_m-\mu_n, 
\end{equation*}
which is a rearrangement of (ii). Similarly, $\alpha_{k+1}\leq B_{k+1}$ gives
\begin{equation*}
\alpha_{k+1}\leq\sum_{m={l}}^{n-1}\mu_m-\sum_{m={l}+1}^{k+1}\beta_m-\sum_{m=k+2}^{n-1}\alpha_m \qquad\forall{l}=1,\ldots,k+1,
\end{equation*}
which is a rearrangement of (iii). Note that we don't use the fact that (iii) holds when ${l}=k+1$. Finally, (iv) follows from the facts that $\{\beta_m\}_{m=1}^n$ is nonincreasing and $\{\beta_m\}_{m=1}^{n} \succeq \{\mu_{m}\}_{m=1}^{n}$, since they imply
\begin{equation*}
\beta_k+\sum_{m={l}+1}^n\beta_m\leq\sum_{m={l}}^n\beta_m\leq\sum_{m={l}}^n\mu_m \qquad\forall{l}=1,\ldots,k,
\end{equation*}
which is a rearrangement of (iv).
\qquad
\end{proof}

We now note that by starting with a sequence $\set{\lambda_{N;m}}_{m=1}^{N}=\set{\lambda_m}_{m=1}^{N}$ that majorizes a given $\set{\mu_m}_{m=1}^{N}$, repeatedly applying Theorem~\ref{theorem.parametrize eigensteps} to construct $\set{\lambda_{n-1;m}}_{m=1}^{n-1}$ from $\set{\lambda_{n;m}}_{m=1}^{n}$ results in a sequence of inner eigensteps that satisfy Definition~\ref{definition.inner eigensteps}.  Conversely, if $\set{\set{\lambda_{n;m}}_{m=1}^{n}}_{n=1}^{N}$ is a valid sequence of inner eigensteps, then for every $n$, (ii) gives $\set{\lambda_{n;m}}_{m=1}^{n-1}\sqsubseteq\set{\lambda_{n;m}}_{m=1}^{n}$, while (ii) and (iii) together imply that $\set{\lambda_{n;m}}_{m=1}^{n}\succeq\set{\mu_m}_{m=1}^{n}$ \`{a} la the discussion at the beginning of Section $3$; as such, any sequence of inner eigensteps can be constructed by repeatedly applying Theorem~\ref{theorem.parametrize eigensteps}.  We now summarize these facts:
\begin{corollary}
\label{corollary.parametrize eigensteps}
Let $\set{\lambda_n}_{n=1}^{N}$ and $\set{\mu_n}_{n=1}^{N}$ be nonnegative and nonincreasing where $\set{\lambda_n}_{n=1}^{N}\succeq\set{\mu_n}_{n=1}^{N}$.  Every corresponding sequence of inner eigensteps $\set{\set{\lambda_{n;m}}_{m=1}^{n}}_{n=1}^{N}$ can be constructed by the following algorithm: Let $\lambda_{N;m}=\lambda_m$ for all $m=1,\dotsc,N-1$; for any $n=N,\dotsc,2$ construct $\set{\lambda_{n-1;m}}_{m=1}^{n-1}$ from $\set{\lambda_{n;m}}_{m=1}^{n}$ by picking $\lambda_{n-1;k}\in[A_{n-1;k},B_{n-1;k}]$ for all $k=n-1,\dotsc,1$, where $A_{n-1;k}$ and $B_{n-1;k}$ are given by~\eqref{eq.lower bound} and~\eqref{eq.upper bound}, respectively.  Moreover, any sequence constructed by this algorithm is indeed a corresponding sequence of inner eigensteps.
\end{corollary}

We now redo Example~\ref{example.5 in 3} to illustrate that Corollary~\ref{corollary.parametrize eigensteps} indeed gives a more systematic way of parametrizing the eigensteps:

\begin{example} 
We wish to parametrize the eigensteps corresponding to UNTFs of 5 vectors in $\mathbb{R}^3$. 
In the end, we will get the same parametrization of eigensteps as in Example~\ref{example.5 in 3}:
\begin{equation} \label{eq.5in3 spectra}
\begin{tabular}{ p{1cm} p{1cm} p{1cm} p{1cm} p{1cm} l}
\ \,$n$&$1$&$2$&$3$&$4$&$5${\smallskip}\\ 
\hline\noalign{\smallskip}
$\lambda_{n;5}$ & & & & &$0${\medskip}\\
$\lambda_{n;4}$ & & & &$0$ &$0${\medskip}\\
$\lambda_{n;3}$ & & &$x$ &$\frac23$ &$\frac{5}3${\medskip}\\
$\lambda_{n;2}$ & &$y$ &$\frac{4}3-x$ &$\frac{5}3$ &$\frac{5}3${\medskip}\\
$\lambda_{n;1}$ &1 &$2-y$ &$\frac{5}3$ &$\frac{5}3$ &$\frac{5}3$\\\end{tabular}
\end{equation}
where $0\leq x\leq\frac23$, $\max\set{\frac13,x}\leq y\leq\min\set{\frac23+x,\frac43-x}$.  In what follows, we rederive the above table one column at a time, in order from right to left, and filling in each column from top to bottom.  First, the desired spectrum of the final Gram matrix gives us that $\lambda_{5,5}=\lambda_{5,4}=0$ and $\lambda_{5,3}=\lambda_{5,2}=\lambda_{5,1}=\frac{5}3$.  Next, we wish to find all $\{\lambda_{4,m}\}_{m=1}^4$ such that $\{\lambda_{4,m}\}_{m=1}^4\sqsubseteq\{\lambda_{5,m}\}_{m=1}^5$ and $\{\lambda_{4,m}\}_{m=1}^4\succeq\{\mu_{m}\}_{m=1}^4$.  To this end, taking $n=5$ and $k=4$, Theorem~\ref{theorem.parametrize eigensteps} gives
\begin{align*} 
\max\{\lambda_{5;5}, \lambda_{5;4}+\lambda_{5;5} - \mu_5\} 
\leq \lambda_{4;4} 
&\leq \min \bigg\{ \lambda_{5;4}, \min_{{l}=1,\dots,4} \bigg\{\sum_{m={l}}^{4} \mu_m - \sum_{m={l}+1}^{4} \lambda_{5;m} \bigg\}\bigg\}, \\
0
=\max\{0, -1 \} 
\leq \lambda_{4;4} 
&\leq \min\{ 0, \tfrac23, \tfrac{4}3, 2, 1\}
=0,
\end{align*}
and so $\lambda_{4;4} = 0$.  For each $k=3,2,1$, the same approach gives $\lambda_{4;3} = \frac23$, $\lambda_{4;2} = \frac{5}3$, and $\lambda_{4;1} = \frac{5}3$.  For the next column, we take $n=4$. Starting with $k=3$, we have
\begin{align*} 
\max\{ \lambda_{4;4}, \lambda_{4;3}+\lambda_{4;4} - \mu_4 \} 
\leq \lambda_{3;3} 
&\leq \min\bigg\{ \lambda_{4;3}, \min_{{l}=1,\dots,3} \bigg\{\sum_{m={l}}^3 \mu_m - \sum_{m={l}+1}^3 \lambda_{4;m} \bigg\}\bigg\}, \\
0
=\max\{0, -\tfrac{1}3\} 
\leq \lambda_{3;3}  
&\leq \min\{\tfrac23, \tfrac23, \tfrac{4}3, 1\}
=\tfrac23.
\end{align*}
Notice that the lower and upper bounds on $\lambda_{3;3}$ are not equal. Since $\lambda_{3;3}$ is our first free variable, we parametrize it: $\lambda_{3;3}=x$ for some $x\in[0,\tfrac23]$. Next, $k=2$ gives
\begin{equation*}
\tfrac43-x
=\max\{\tfrac23, \tfrac{4}3 -x\} 
\leq \lambda_{3;2} 
\leq \min\{\tfrac{5}3, \tfrac{4}3-x, 2-x\}
=\tfrac43-x,
\end{equation*}
and so $\lambda_{3;2}=\tfrac43-x$.
Similarly, $\lambda_{3;1} = \frac{5}3$. 
Next, we take $n=3$ and $k=2$:
\begin{equation*}
\max \{x, \tfrac13 \} 
\leq \lambda_{2;2}
\leq \min \{\tfrac{4}3-x, \tfrac23+x, 1\}.
\end{equation*}
Note that $\lambda_{2;2}$ is a free variable; we parametrize it as $\lambda_{2;2}=y$ such that
\begin{equation*}
y\in[\tfrac13,\tfrac23+x] \mbox{ if } x\in[0,\tfrac13],\qquad
y\in[x,\tfrac43-x] \mbox{ if } x\in[\tfrac13,\tfrac23].
\end{equation*}
Finally, $\lambda_{2,1} = 2 - y$ and $\lambda_{1,1} = 1$. 
\end{example}

We conclude by giving a complete constructive solution to Problem~\ref{problem.main}, that is, the problem of constructing every frame of a given spectrum and set of lengths.  Recall from the introduction that it suffices to prove Theorem~\ref{theorem.Step A redone}:

\emph{Proof of Theorem~\ref{theorem.Step A redone}:}
We first show that such an $F$ exists if and only if we have $\set{\lambda_m}_{m=1}^{M}\cup\{0\}_{m=M+1}^N\succeq\set{\mu_n}_{n=1}^{N}$.  Though this may be quickly proven using the Schur-Horn Theorem---see the discussion at the beginning of Section $2$---it also follows from the theory of this paper.  In particular, if such an $F$ exists, then Theorem~\ref{theorem.necessity and sufficiency of eigensteps} implies that there exists a sequence of outer eigensteps corresponding to $\set{\lambda_m}_{m=1}^{M}$ and $\set{\mu_n}_{n=1}^{N}$; by Theorem~\ref{theorem.inner vs outer}, this implies that there exists a sequence of inner eigensteps corresponding to $\set{\lambda_m}_{m=1}^{M}\cup\{0\}_{m=M+1}^N$ and $\set{\mu_n}_{n=1}^{N}$; by the discussion at the beginning of Section $3$, we necessarily have $\set{\lambda_m}_{m=1}^{M}\cup\{0\}_{m=M+1}^N\succeq\set{\mu_n}_{n=1}^{N}$.  Conversely, if $\set{\lambda_m}_{m=1}^{M}\cup\{0\}_{m=M+1}^N\succeq\set{\mu_n}_{n=1}^{N}$, then Top Kill (Theorem~\ref{theorem.top kill}) constructs a corresponding sequence of inner eigensteps, and so Theorem~\ref{theorem.inner vs outer} implies that there exists a sequence of outer eigensteps corresponding to $\set{\lambda_m}_{m=1}^{M}$ and $\set{\mu_n}_{n=1}^{N}$, at which point Theorem~\ref{theorem.necessity and sufficiency of eigensteps} implies that such an $F$ exists. 

For the remaining conclusions, note that in light of Theorem~\ref{theorem.necessity and sufficiency of eigensteps}, it suffices to show that every valid sequence of outer eigensteps (Definition~\ref{definition.outer eigensteps}) satisfies the bounds of Step~A of Theorem~\ref{theorem.Step A redone}, and conversely, that every sequence constructed by Step~A is a valid sequence of outer eigensteps.  Both of these facts follow from the same two results.  The first is Theorem~\ref{theorem.inner vs outer}, which establishes a correspondence between every valid sequence of outer eigensteps for $\set{\lambda_m}_{m=1}^{M}$ and $\set{\mu_n}_{n=1}^{N}$ with a valid sequence of inner eigensteps for $\set{\lambda_m}_{m=1}^{M}\cup\set{0}_{m=M+1}^{N}$ and $\set{\mu_n}_{n=1}^{N}$ and vice versa, the two being zero-padded versions of each other.  The second relevant result is Corollary~\ref{corollary.parametrize eigensteps}, which characterizes all such inner eigensteps in terms of the bounds~\eqref{eq.lower bound} and~\eqref{eq.upper bound} of Theorem~\ref{theorem.parametrize eigensteps}.  In short, the algorithm of Step~A is the outer eigenstep version of the application of Corollary~\ref{corollary.parametrize eigensteps} to $\set{\lambda_m}_{m=1}^{M}\cup\set{0}_{m=M+1}^{N}$; one may easily verify that all discrepancies between the statement of Theorem~\ref{theorem.Step A redone} and Corollary~\ref{corollary.parametrize eigensteps} are the result of the zero-padding that occurs in the transition from inner to outer eigensteps.
\qquad
\endproof

\section*{Acknowledgments}
This work was supported by NSF DMS 1042701, NSF CCF 1017278, AFOSR F1ATA01103J001, AFOSR F1ATA00183G003 and the A.~B.~Krongard Fellowship.  The views expressed in this article are those of the authors and do not reflect the official policy or position of the United States Air Force, Department of Defense, or the U.S.~Government.


\newpage
 
\begin{thebibliography}{WW}

\bibitem{AntezanaMRS:07}
{\sc J.~Antezana, P.~Massey, M.~Ruiz and D.~Stojanoff},
{\it The Schur-Horn theorem for operators and frames with prescribed norms and frame operator},
Illinois\ J.\ Math., 51 (2007), pp.~537--560.

\bibitem{BatsonSS:11}
{\sc J.~Batson, D.~A.~Spielman and N.~Srivastava},
{\it Twice-Ramanujan sparsifiers},
SIAM J.\ Comput., to appear.

\bibitem{BendelM:78}
{\sc R.~B.~Bendel and M.~R.~Mickey},
{\it Population correlation matrices for sampling experiments},
Comm.\ Statist.\ Simulation Comput., 7 (1978), pp.~163--182.

\bibitem{BodmannC:10}
{\sc B.~G.~Bodmann and P.~G.~Casazza},
{\it The road to equal-norm Parseval frames},
J.\ Funct.\ Anal., 258 (2010), pp.~397--420.

\bibitem{CahillFMPS:11}
{\sc J.~Cahill, M.~Fickus, D.~G.~Mixon, M.~J.~Poteet and N.~Strawn},
{\it Constructing finite frames of a given spectrum and set of lengths},
submitted, arXiv:1106.0921.

\bibitem{CasazzaFM:11}
{\sc P.~G.~Casazza, M.~Fickus and D.~G.~Mixon},
{\it Auto-tuning unit norm tight frames},
Appl.\ Comput.\ Harmon.\ Anal., to appear.

\bibitem{CasazzaFMWZ:11}
{\sc P.~G.~Casazza, M.~Fickus, D.~G.~Mixon, Y.~Wang and Z.~Zhou},
{\it Constructing tight fusion frames},
Appl.\ Comput.\ Harmon.\ Anal., 30 (2011), pp.~175--187.

\bibitem{CasazzaK:03}
{\sc P.~G.~Casazza and J.~Kova{\v c}evi{\'c}},
{\it Equal-norm tight frames with erasures},
Adv.\ Comp.\ Math., 18 (2003), pp.~387--430.

\bibitem{CasazzaL:06}
{\sc P.~G.~Casazza and M.~Leon},
{\it Existence and construction of finite tight frames},
J.\ Comput.\ Appl.\ Math., 4 (2006), pp.~277--289.

\bibitem{ChanLi:83}
{\sc N.~N.~Chan and K.-H.~Li},
{\it Diagonal elements and eigenvalues of a real symmetric matrix},
J.\ Math.\ Anal.\ Appl., 91 (1983), pp.~562--566.

\bibitem{Chu:95}
{\sc M.~T.~Chu},
{\it Constructing a Hermitian matrix from its diagonal entries and eigenvalues},
SIAM J.\ Matrix Anal.\ Appl., 16 (1995), pp.~207--217.

\bibitem{DavisH:00}
{\sc P.~I.~Davies and N.~J.~Higham},
{\it Numerically stable generation of correlation matrices and their factors},
BIT, 40 (2000), pp.~640--651.

\bibitem{DhillonHST:05}
{\sc I.~S.~Dhillon, R.~W.~Heath, M.~A.~Sustik and J.~A.~Tropp},
{\it Generalized finite algorithms for constructing Hermitian matrices with prescribed diagonal and spectrum},
SIAM\ J.\ Matrix Anal.\ Appl., 27 (2005), pp.~61--71.

\bibitem{DykemaS:06}
{\sc K.~Dykema and N.~Strawn},
{\it Manifold structure of spaces of spherical tight frames},
Int.\ J.\ Pure Appl.\ Math., 28 (2006), pp.~217--256.

\bibitem{FickusMP:11}
{\sc M.~Fickus, D.~G.~Mixon and M.~J.~Poteet},
{\it Frame completions for optimally robust reconstruction},
Proc.\ SPIE., 8138 (2011), to appear, arXiv:1107.1912. 

\bibitem{GoyalKK:01}
{\sc V.~K.~Goyal, J.~Kova{\v c}evi{\'c} and J.~A.~Kelner},
{\it Quantized frame expansions with erasures},
Appl.\ Comput.\ Harmon.\ Anal., 10 (2001), pp.~203--233.

\bibitem{GoyalVT:98}
{\sc V.~K.~Goyal, M.~Vetterli and N.~T.~Thao},
{\it Quantized overcomplete expansions in ${\mathbb R}^N$: Analysis, synthesis, and algorithms},
IEEE\ Trans.\ Inform.\ Theory, 44 (1998), pp.~16--31.

\bibitem{Higham:89}
{\sc N.~J.~Higham}, 
{\it Matrix nearness problems and applications},
in Applications of matrix theory, M.~J.~C.~Gover and S.~Barnett eds., Oxford University Press, 1989, pp. 1--27.

\bibitem{HolmesP:04}
{\sc R.~B.~Holmes and V.~I.~Paulsen},
{\it Optimal frames for erasures},
Linear~Algebra Appl., 377 (2004), pp.~31--51.

\bibitem{Horn:54}
{\sc A.~Horn},
{\it Doubly stochastic matrices and the diagonal of a rotation matrix},
Amer.\ J.\ Math., 76 (1954), pp.~620--630.

\bibitem{HornJ:85}
{\sc R.~A.~Horn and C.~R.~Johnson},
{\it Matrix Analysis},
Cambridge University Press, Cambridge, 1985.

\bibitem{LeiteRT:99}
{\sc R.~S.~Leite, T.~R.~W.~Richa and C.~Tomei},
{\it Geometric proofs of some theorems of Schur-Horn type},
Linear Algebra Appl., 286 (1999), pp.~149--173.

\bibitem{MasseyR:08}
{\sc P.~Massey and M.~Ruiz},
{\it Tight frame completions with prescribed norms},
Sampl.\ Theory Signal Image Process., 7 (2008), pp.~1--13.

\bibitem{Schur:23}
{\sc I.~Schur},
{\it \"{U}ber eine klasse von mittelbildungen mit anwendungen auf die determinantentheorie},
Sitzungsber.\ Berl.\ Math.\ Ges., 22 (1923), pp.~9--20.

\bibitem{Strawn:11}
{\sc N.~Strawn},
{\it Finite frame varieties: nonsingular points, tangent spaces, and explicit local parameterizations},
J.\ Fourier\ Anal.\ Appl., to appear.

\bibitem{TroppDH:04}
{\sc J.~A.~Tropp, I.~S.~Dhillon and R.~W.~Heath},
{\it Finite-step algorithms for constructing optimal CDMA signature sequences},
IEEE Trans.\ Inform.\ Theory, 50 (2004), pp.~2916--2921.

\bibitem{ViswanathA:99}
{\sc P.~Viswanath and V.~Anantharam},
{\it Optimal sequences and sum capacity of synchronous {CDMA} systems},
IEEE\ Trans.\ Inform.\ Theory, 45 (1999), pp.~1984--1991.

\end{thebibliography}
\end{document}